\documentclass[11pt]{amsart}
\usepackage[parfill]{parskip}    
\usepackage{fullpage}
\usepackage{graphicx}
\usepackage{amssymb,amsmath,latexsym}
\usepackage{tikz}
\usetikzlibrary{arrows,decorations.markings,decorations.pathreplacing,calc,matrix,intersections}
\tikzset{->-/.style={decoration={markings,mark=at position #1 with {\arrow{>}}},postaction={decorate}}}

\usepackage{epstopdf}
\usepackage{ifpdf}
\usepackage{color}
\usepackage{float}
\usepackage{amscd,amsmath}
\usepackage{amssymb,amsmath,latexsym,color,enumerate,tikz}
\usepackage[all]{xypic}       
\usepackage[varg]{pxfonts}
\usepackage{amsmath,amsthm,amsfonts,amssymb,fancyhdr,graphics,relsize,tikz-cd,mathtools,faktor,tikz,changepage}
\usepackage{graphicx}
\usepackage{placeins}

\usepackage{dsfont}
\definecolor{red}{rgb}{1,0,0} 

 \definecolor{darkgreen}{rgb}{0, .7, 0}

 \definecolor{purple}{rgb}{.7, 0, 1}

\tikzset{mynode/.style={draw,circle,fill=black,inner sep=2pt,outer sep=0.5pt}}

\newtheorem{theorem}{Theorem}[section]
\newtheorem*{theorem*}{Theorem}
\newtheorem*{cor*}{Corollary}

\newtheorem*{mainthm}{Theorem \ref{Main theorem}}

\newtheorem*{thmM1}{Theorem \ref{Theorem case 1}}
\newtheorem*{lemma*}{Lemma}
\newtheorem{proposition}[theorem]{Proposition}
\newtheorem{lemma}[theorem]{Lemma}
\newtheorem{corollary}[theorem]{Corollary}

\theoremstyle{definition}
\newtheorem{definition}[theorem]{Definition}

\newtheorem{question}[theorem]{Question}
\newtheorem{conj}[theorem]{Conjecture}

\theoremstyle{remark}
\newtheorem{remark}[theorem]{Remark}

\usepackage{hyperref}

\begin{document}
\title{Subgroups of the direct product of graphs of groups with free abelian vertex groups}
\author{Montserrat Casals-Ruiz, Jone Lopez de Gamiz Zearra}
\maketitle
\begin{abstract}
    A result of Baumslag and Roseblade states that a finitely presented subgroup of the direct product of two free groups is virtually a direct product of free groups. In this paper we generalise this result to the class of cyclic subgroup separable graphs of groups with free abelian vertex groups and cyclic edge groups. More precisely, we show that a finitely presented subgroup of the direct product of two groups in this class is virtually $H$-by-(free abelian), where $H$ is the direct product of two groups in the class. In particular, our result applies to 2-dimensional coherent right-angled Artin groups and residually finite tubular groups. Furthermore, we show that the multiple conjugacy problem and the membership problem are decidable for finitely presented subgroups of the direct product of two $2$-dimensional coherent RAAGs.
\end{abstract}

\section{Introduction}
Since the 60's it is well known that finitely generated subgroups of the direct product of non-abelian free groups are very complex and, in particular, most algorithmic problems such as the conjugacy, the isomorphism and the membership problem are undecidable (see \cite{Mihailova, Miller2}).

In \cite{Grun}, Grunewald showed that the subgroups with this complex behaviour are not finitely presented. Remarkably, in \cite{Roseblade} (see \cite{something1, Miller, Short} for alternative proofs), Baumslag and Roseblade clarified the situation and proved that finitely presented subgroups of the direct product of two free groups have a very tame structure - they are virtually the direct product of two free groups. The study of finitely presented subgroups of the direct product of arbitrarily many free groups (and more generally, limit groups over free groups) was conducted in a series of papers that culminated in \cite{Bridson3} where the authors prove that these subgroups also have a tame structure and that the main algorithmic problems are decidable. 

\emph{Right-angled Artin groups (RAAGs)} are defined by presentations where the relations are commutation of some pairs of generators and so it extends the class of (direct products of) finitely generated free groups. In view of the results about subgroups of the direct product of free groups, one may wonder if finitely presented subgroups of RAAGs have a tame structure and, in particular, if the main algorithmic problems are decidable in that class. Unfortunately, this is not the case as Bridson showed in \cite{Bridson4} that there is a right-angled Artin group $A$ and a finitely presented subgroup $S < A\times A$ for which the conjugacy and the membership problems are undecidable. 

This work is part of a series that aims at describing the structure of finitely presented subgroups of the direct product of (limit groups over) coherent RAAGs, that is, the structure of finitely presented residually coherent RAAGs, and at showing that the main algorithmic problems are decidable for this class. This programme was carried over for the subclass of RAAGs whose finitely generated subgroups are again RAAGs in \cite{Zearra}. 

In this paper, we begin studying the class of \emph{$2$-dimensional coherent RAAGs}. More precisely, we generalise Baumslag and Roseblade's result for free groups and we describe the structure of finitely presented subgroups of the direct product of two $2$-dimensional coherent RAAGs:

\begin{theorem*}
Let $S$ be a finitely presented subgroup of the direct product of two $2$-dimensional coherent RAAGs. Then, $S$ is virtually $H$-by-(free abelian), where $H$ is the direct product of two subgroups of $2$-dimensional coherent RAAGs.
\end{theorem*}

Furthermore, we show that these finitely presented subgroups have a good algorithmic behaviour. Namely, we proof the following:
 
\begin{cor*}
Finitely presented subgroups of the direct product of two $2$-dimensional coherent RAAGs have decidable multiple conjugacy problem and membership problem.
\end{cor*}

This corollary shows that Bridson's example of a right-angled Artin group $A$ and an algorithmically bad finitely presented subgroup of $A\times A$ is not $2$-dimensional coherent (and we conjecture that $A$ cannot be coherent).

In fact, our results apply to a wider class of groups, the class $\mathcal{A}$, which is the $Z \ast$- closure (see Section~\ref{Section 1} for the definition) of the class of cyclic subgroup separable graphs of groups with free abelian vertex groups and cyclic edge groups. This class contains $2$-dimensional coherent RAAGs and residually finite tubular groups among others. Recall that a \emph{tubular group} is a finitely generated graph of groups with $\mathbb{Z}^2$ vertex groups and $\mathbb{Z}$ edge groups.

\begin{mainthm}
Let $S$ be a finitely presented subgroup of $G_{1}\times G_{2}$ where $G_{1}, G_{2} \in \mathcal{A}$. Then, $S$ is virtually $H$-by-(free abelian), where $H$ is the direct product of two groups in $\mathcal{A}$.

Furthermore, $S$ is virtually the kernel of a homomorphism $f$ from $H_1\times H_2$ to $\mathbb R$, for some $H_1, H_2\in \mathcal A$. More precisely, $H$ is equal to $L_1 \times L_2$, where $L_i = S\cap G_i$, $i\in \{1,2\}$ and either
\begin{itemize}
    \item $L_1 \times L_2 <_{fi} S <_{fi} G_1 \times G_2$, or
    \item $S$ is virtually the kernel $\ker f$ where $f\colon H_1 \times H_2 \mapsto \mathbb Z$ for some $H_i \in \mathcal A$, $i\in \{1,2\}$, or
    \item $S$ is virtually the kernel $\ker f$ where $f\colon H_1 \times H_2 \mapsto \mathbb{Z}^2$ for some $H_i\in \mathcal A$, $i\in \{1,2\}$. In this case, $L_i$ is the free product of finitely generated groups in $\mathcal A$ for $i\in \{1,2\}$.
\end{itemize}
\end{mainthm}

In \cite{Bestvina}, Bestvina and Brady examined the finiteness properties of the kernels of homomorphisms from RAAGs to $\mathbb Z$ using Morse theory. As a consequence of our main result, we obtain that the finitely presented subgroup $S$ is virtually a kernel of a homomorphism from the direct product of two groups in our class to $\mathbb Z^n$ and so we expect that Morse theory can also be useful to further study the finiteness properties of these subgroups.

The abelian factor in the description of $S$ is directly related to the edge groups of the decomposition of the groups in our class. In particular, if we consider graphs of groups in $\mathcal A$ with trivial edge groups, we deduce the following theorem and recover the result of Baumslag and Roseblade for direct products of free groups:

\begin{thmM1}
Let $\mathcal{A}^\prime$ be the subclass of $\mathcal{A}$ containing the groups which have a non-trivial free product decomposition and let $S$ be a finitely presented subgroup of the direct product of two groups in the class $\mathcal{A}^\prime$. Then, $S$ is virtually the direct product of two subgroups in $\mathcal{A}^\prime$.
\end{thmM1}

We require the graphs of groups to be cyclic subgroup separable. Recall that a group $G$ is \emph{cyclic subgroup separable} if for each cyclic subgroup $C <G$ and each $h\in G \setminus C$, there is a finite quotient $Q$ of $G$ and a homomorphism $\pi \colon G \mapsto Q$ such that $\pi(h) \notin \pi(C)$. Therefore, cyclic subgroup separability is a residual property that generalises residual finiteness. Although cyclic subgroup separability is a property of algebraic nature, we believe that it carries consequences on the geometry of the groups in the class $\mathcal A$. Indeed, we believe that the groups in the class $\mathcal A$ are CAT(0) and in fact we conjecture that they are virtually compact special. More precisely, we propose the following:

\begin{conj}
Let $G$ be a finitely generated graph of groups with free abelian vertex groups and with cyclic edge groups. Then, the following are equivalent:

\begin{itemize}
    \item $G$ is cyclic subgroup separable;
    \item $G$ has a finite index subgroup with isolated edge groups; \footnote{In \cite{Hoda}, the authors use the terminology virtually primitive.}
    \item $G$ is virtually compact special.
\end{itemize} 
Furthermore, if $G$ is freely indecomposable, then $G$ is cyclic subgroup separable if and only if it is residually finite and non-solvable.
\end{conj}

Recall that a subgroup $H<G$ is \emph{isolated} if, for each $g\in G$ and each $n\in \mathbb N\setminus \{0\}$, $g^n \in H$ implies $g\in H$. 

We review some of the results that support the conjecture. In the case of \emph{Generalised Baumslag-Solitar groups}, we have that they are cyclic subgroup separable if and only if they are unimodular if and only if they have a finite index subgroup with isolated edge groups if and only if they are non-solvable and residually finite (see \cite[Corollary 7.7]{Levitt}). Unimodular Generalised Baumslag-Solitar groups are precisely the virtually compact special ones. Therefore, the conjecture holds for Generalised Baumslag-Solitar groups.

In \cite{Hoda}, the authors prove that for tubular groups the properties of being cyclic subgroup separable, of being residually finite and of having a finite index subgroup with isolated edge groups are equivalent. Furthermore, in \cite{Button}, Button proves that a tubular group whose underlying graph is a tree is virtually special.

For the general case, any finitely generated graph of groups such that the underlying graph is a tree, with free abelian vertex groups and with cyclic edge groups is cyclic subgroup separable (see \cite[Theorem 3.4]{Evans}). As in the case of tubular groups, we expect these groups to be virtually special.

If the underlying graph has only one loop, that is, if $G= \langle T, s \mid a^s = b\rangle$ and $T$ is a tree of free abelian vertex groups and cyclic edge groups, then the group $G$ is cyclic subgroup separable if and only if it has a finite index subgroup which is a graph of groups with isolated edge groups. Furthermore, the later is equivalent to either $\langle a \rangle \cap \langle b \rangle = \{1\}$ or $a=c^k$ and $b=c^{\pm k}$ for some $c\in G$ and $k\in \mathbb{Z}$ (see \cite[Corollay 3.5]{GKim}).

In addition, for graphs of groups with free abelian vertex groups and cyclic edge groups, having isolated edge groups is a sufficient condition in order to be cyclic subgroup separable (see \cite[Theorem 3.7]{GKim}).

As mentioned above, non-unimodular (Generalised) Baumslag-Solitar groups are not cyclic subgroup separable, and so they do not belong to our class. We therefore ask the following:

\begin{question}
What is the structure of finitely presented subgroups of the direct product of two non-unimodular (Generalised) Baumslag-Solitar groups?
\end{question}

Our results apply to $2$-dimensional coherent RAAGs but we conjecture that they should extend to all coherent RAAGs and, more generally, to cyclic subgroup separable graphs of groups with free abelian vertex groups. Interestingly enough, in order to generalise the results from $2$-dimensional coherent RAAGs to all coherent RAAGs, it would suffice to prove a generalised version of It\^o's theorem (see \cite{Ito}). It\^o's theorem states that if a group $Q$ is the product $AB$ of two abelian subgroups $A$ and $B$, then $Q$ is metabelian. In our case, we would need to generalise this result to finitely many cosets of a product of two abelian groups, namely, if
$$Q= A_{1}A_{2} \dot{\cup} A_{1}A_{2}a_{1} \dot{\cup} \cdots \dot{\cup} A_{1}A_{2} a_{r},
$$
where $A_{1}$ and $A_{2}$ are abelian subgroups, then one would like to conclude that the group $Q$ is virtually metabelian. We did not succeed in proving this general statement and so we formulate it as a question:

\begin{question}
Can It\^o's theorem be extended in the presence of cosets? That is, if $Q$ is a group such that
$$Q= A_{1}A_{2} \dot{\cup} A_{1}A_{2}a_{1} \dot{\cup} \cdots \dot{\cup} A_{1}A_{2} a_{r},
$$ 
where $A_{1}$ and $A_{2}$ are abelian subgroups, is $Q$ virtually metabelian?
\end{question}

When one restricts to groups in our class, that is with cyclic edge groups, the abelian subgroups $A_{1}$ and $A_{2}$ in the decomposition above are actually cyclic groups. This allows us to further reduce the problem to the case when a Baumslag-Solitar group admits such a decomposition, that is when a Baumslag-Solitar group can be covered by finitely many cosets of the product of two cyclic subgroups. In this case, we use the structure of Baumslag-Solitar groups to show that this is only possible when the Baumslag-Solitar group is abelian.

We also notice that for graphs of groups with free abelian edge groups the condition on cyclic subgroup separability is more subtle than in the case of cyclic edge groups as there are examples of amalgamated free products of free abelian groups that are not residually finite (see \cite[Remark 11.6]{Leary}).

In relation to algorithmic problems, it is known that the isomorphism problem for finitely presented subgroups of the direct product of two free groups is decidable (see \cite{Bridson3}), although for finitely presented subgroups of the direct product of finitely many free groups it is still open. In the case of $2$-dimensional coherent RAAGs, we formulate it as a question:

\begin{question}
Is the isomorphism problem for finitely presented subgroups of the direct product of two $2$-dimensional coherent RAAGs decidable?
\end{question}

The paper is organized as follows. In Section~\ref{Section 1}, we introduce the class of groups that we will study and describe some properties of these groups. 

In Section~\ref{Section 2}, we review Miller's proof for free groups. The idea is to first show that if $S$ is a finitely presented subgroup of the direct product of two free groups $F_{1}$ and $F_{2}$, then the subgroups $L_{i}= S \cap F_{i}$ are finitely generated, and after that use the fact that non-trivial finitely generated  normal subgroups of free groups have finite index to conclude that the direct product $L_{1}\times L_{2}$ has finite index in $S$.

When considering a finitely presented subgroup $S$ of the direct product of two $2$-dimensional coherent RAAGs, say $G_{1}$ and $G_{2}$, the intersections $S\cap G_{1}$ and $S \cap G_{2}$ are not necessarily finitely generated. Furthermore, non-trivial finitely generated normal subgroups of coherent RAAGs do not need to be of finite index. Indeed, coherent RAAGs fiber, that is, they admit non-trivial epimorphisms onto $\mathbb{Z}$ with finitely generated kernel.

We address these issues in Section~\ref{Section 3}. We show that although the subgroup $L_{i}$ may not be finitely generated, a cyclic extension of $L_{i}$ is (see Proposition~\ref{Proposition 1}). We also characterise finitely generated normal subgroups $N$ of a group $G$ in $\mathcal{A}$: either $N$ is in the center of the group $G$ or $G \slash N$ is virtually abelian (see Proposition~\ref{Proposition 0}).

Finally, in Section~\ref{Section 4}, the main result is proved. We show that the quotient $Q=G_{i} \slash L_{i}$ is covered by finitely may cosets of the product of two cyclic subgroups and that this covering lifts to a covering of a Baumslag-Solitar group. We use the structure of Baumslag-Solitar groups to deduce that it is free abelian and conclude that $Q$ is virtually free abelian.

\section{The class of groups $\mathcal A$}
\label{Section 1}

Given a class of groups, there is a natural way to construct other groups using operations such as taking free products or adding center. 

\begin{definition}
Let $\mathcal{C}$ be a class of groups. The \emph{$Z \ast$-closure of $\mathcal{C}$}, denoted by $Z \ast (\mathcal{C})$, is the union of classes $(Z \ast (\mathcal{C}))_{k}$ defined recursively as follows. At level $0$, the class $(Z \ast (\mathcal{C}))_{0}$ is the class $\mathcal{C}$. A group $G$ lies in $(Z \ast (\mathcal{C}))_{k}$ for $k\geq 1$ if and only if
\[ G \cong \mathbb{Z}^m \times (G_{1}\ast \cdots \ast G_{n}),\] where $m\in \mathbb{N}\cup \{0\}$ and the groups $G_{i}$ belong to $(Z \ast (\mathcal{C}))_{k-1}$ for all $i\in \{1,\dots,n\}$.
\end{definition}

If the class $\mathcal C$ is $\mathbb{Z} $, then Droms proved in \cite{Droms} that the $Z \ast$-closure of $\mathbb Z$ is precisely the family of RAAGs with the property that all their finitely generated subgroups are again RAAGs. 

The class of groups $\mathcal A$ is defined as the $Z \ast$-closure of the class of groups $\mathcal G$, defined as follows.

\begin{definition}[Class $\mathcal A$ and class $\mathcal G$]
Let $\mathcal{G}$ be the class of cyclic subgroup separable graphs of groups with free abelian vertex groups and cyclic edge groups. 

The class of groups $\mathcal A$ is the $Z \ast$-closure of the class $\mathcal G$.
\end{definition}

\begin{definition}
Let $G$ be a group in the class $\mathcal{G}$. Any splitting of $G$ as a graph of groups with free abelian vertex groups and cyclic edge groups is called a \emph{standard splitting} of $G$.
\end{definition}

Notice that cyclic subgroup separability extends from the class $\mathcal G$ to the class $\mathcal A$; that is, if $G$ is a group in $\mathcal{A}$, then $G$ is cyclic subgroup separable. Indeed, by assumption, groups in the class $\mathcal G$ are cyclic subgroup separable. Furthermore, cyclic subgroup separability is closed under taking free products (see \cite[Lemma 2.3]{Evans}) and adding center: $G$ is cyclic subgroup separable if and only if so is $\mathbb{Z}^n\times G$, $n\in \mathbb{N}\cup \{0\}$.

The class of groups $\mathcal A$ is closed under taking subgroups: if $G\in \mathcal A$ and $H<G$, then $H\in \mathcal A$. In particular, subgroups of groups in $\mathcal A$, modulo the center, have a non-trivial free product decomposition or split over an infinite cyclic subgroup. More precisely, we have the following

\begin{lemma}\label{Lemma Subgroups}
Let $G$ be a group of $\mathcal{A}$ and let $H$ be a finitely generated subgroup of $G$. Then, $H \in \mathcal A$ and $H \cong \mathbb Z^n \times H^\prime$ for $n\in \mathbb{N} \cup \{0\}$, where
\begin{itemize}
\item $H^\prime$ is a non-trivial free product of groups in $\mathcal A$, and so in particular, it is residually finite, or\\[3pt]
\item $H^\prime$ belongs to $\mathcal G$.
\end{itemize}
\end{lemma}

\begin{proof}
The proof is by induction on the level of the group $G\in \mathcal A$. The class $\mathcal G$ is closed under subgroups: if $G\in \mathcal G$, then it is a cyclic subgroup separable graph of groups with free abelian vertex groups and cyclic edge groups and so is any finitely generated subgroup $H$.

Assume now that the level of $G$ is $k>0$. Then, $G$ is of the form $\mathbb{Z}^m \times (G_{1}\ast \cdots \ast G_{n})$ for some $m\in \mathbb{N}\cup \{0\}$ and $G_i \in \mathcal A$ of level less than $k$. Then, there is a short exact sequence
\[ 
\begin{tikzcd}
1 \arrow{r}  & \mathbb{Z}^m \arrow{r} & G \arrow{r}{p} & G_{1}\ast \cdots \ast G_{n} \arrow{r} & 1.
\end{tikzcd} 
\]
It follows that $H$ is of the form $( H \cap \mathbb{Z}^m) \times p(H)$ and $p(H)$ is a subgroup of $G_{1}\ast \cdots \ast G_{n}$. If $n=1$, $p(H)$ is a group of $\mathcal{G}$ and so the first alternative of the statement holds. If $n>1$, then the result follows from the Kurosh subgroup theorem and the induction hypothesis.
\end{proof}

The class $\mathcal{A}$ contains interesting families of groups. Of special interest to us are $2$-dimensional coherent RAAGs and residually finite tubular groups. We next recall the definitions of these families of groups and we show that they indeed belong to the class $\mathcal A$.

Recall that given a finite simplicial graph $X$ with vertex set $V(X)$ and edge set $E(X)$, the corresponding \emph{right-angled Artin group} (\textit{RAAG}), denoted by $GX$, is given by the following presentation:
$$
GX= \langle V(X) \mid  xy=yx \iff (x,y) \in E(X) \rangle.
$$
To each right-angled Artin group $GX$, we can associate a $CAT(0)$ cube complex, called the \emph{Salvetti complex}. The \emph{dimension} of the RAAG $GX$ is defined as the (topological) dimension of the Salvetti complex which coincides with the number of vertices of the largest complete full subgraph of $X$. In particular, $GX$ is $2$-dimensional if and only if the underlying graph $X$ is triangle-free. See \cite{Charney} for a general survey of right-angled Artin groups.

Recall that a group is \emph{coherent} if all its finitely generated subgroups are also finitely presented. Carls Droms characterized coherent RAAGs in terms of the defining graph as follows (see \cite{Droms3}): the group $GX$ is coherent if and only if $X$ does not have induced cycles of length greater than $3$. In particular, $2$-dimensional coherent RAAGs are defined by graphs which are forests, that is, they are free products of tree groups, where by a \emph{tree group} we mean a right-angled Artin group whose defining graph is a tree.

The RAAG defined by the path with $4$ vertices $P_4$ and which, abusing the notation, we also denote by $P_4$ is the group given by the presentation:
\[
\langle a,b,c,d \mid ab=ba, bc=cb, cd=dc\rangle.
\]
Note that $P_{4}$ admits the following splitting as a graph of groups with free abelian vertex groups and cyclic edge groups:
\[P_{4}=\langle a,b\rangle \ast_{\langle b \rangle} \langle b,c \rangle  \ast_{\langle c \rangle} \langle c,d \rangle,\]

and it is cyclic subgroup separable (see, for instance, \cite[Theorem 3.4]{Evans}). Therefore, $P_{4} \in \mathcal{G}$. 

The RAAG $P_4$ plays an important role in the theory of $2$-dimensional coherent RAAGs as it serves as universe for them. Indeed, in \cite{Kim} it is shown that any $2$-dimensional coherent RAAG embeds in the group $P_{4}$. In particular, since $P_4$ belongs to $\mathcal G$ and the class is closed under subgroups, we deduce that $2$-dimensional coherent RAAGs also belong to $\mathcal{G}$.

In fact, all tree groups are examples of \emph{residually finite tubular groups}. Tubular groups are finitely generated groups that split as graphs of groups with $\mathbb{Z}^2$ vertex groups and $\mathbb{Z}$ edge groups. Despite their simple definition, they have a surprisingly rich source of diverse behaviour. Tubular groups provide examples of finitely generated $3$-manifold groups that are not subgroup separable; of free-by-cyclic groups that do not act properly and semi-simply on a CAT(0) space; of groups that are CAT(0) but not Hopfian, etc (see \cite{Hoda} and references there). In the same paper, the authors give a characterisation of cyclic subgroup separable tubular groups and prove that a tubular group is cyclic subgroup separable if and only if it is residually finite if and only if it is virtually primitive. Recall that a tubular group $G$ is \emph{primitive} if each edge group is a maximal cyclic subgroup of its vertex groups. From this characterisation we have that residually finite tubular groups belong to the class $\mathcal A$.

\section{Miller's proof and counterexamples in tree groups}
\label{Section 2}

Baumslag and Roseblade's result states that given $F_{1}$ and $F_{2}$ two finitely generated free groups and $S$ a finitely presented subgroup of $F_{1}\times F_{2}$, then $S$ is free or $S$ is virtually the direct product of two free groups.

We now briefly sketch Miller's strategy to highlight the relevant properties of free groups that are used in his proof of the aforementioned result.  Recall that a subgroup of a direct product is called a \emph{subdirect product} if its projection to each factor is surjective.

First of all, we can reduce to the case when $S$ is a subdirect product. Indeed, if we consider the projection maps $p_{1}\colon S \mapsto F_{1}$ and $p_{2} \colon S \mapsto F_{2}$, $p_{i}(S)$ is a finitely generated free group for $i\in \{1,2\}$, so we can assume that the projection maps are surjective.

Let us define $L_{i}$ to be $S\cap F_{i}$, $i\in \{1,2\}$. That is, \[L_{1}=\{s\in F_{1} \mid (s,1)\in S\} \quad \text{and} \quad L_{2}=\{s\in F_{2} \mid (1,s)\in S\}.\]

Observe that there is a short exact sequence
\[ 
\begin{tikzcd}
1 \arrow{r}  & L_{2} \arrow{r} & S \arrow{r}{p_{1}} & F_{1} \arrow{r} & 1.
\end{tikzcd} 
\]

If $L_{2}$ is trivial, $S$ is isomorphic to $F_{1}$ and so $S$ is free. A symmetric argument applies if $L_{1}$ is trivial.

Now assume that $L_{1}$ and $L_{2}$ are both non-trivial. Miller then proves, by using Marshall Hall's theorem for free groups, that for $i\in \{1,2\}$, $S$ is virtually an HNN extension with associated subgroup $L_{i}$ and since $S$ is finitely presented, $L_{i}$ needs to be finitely generated (see \cite[Lemma 2, Theorem 1]{Miller}).

From here one deduces that for $i\in \{1,2\}$, $L_{i}$ is a non-trivial finitely generated normal subgroup of $F_{i}$, so $L_{i}$ has finite index in $F_{i}$. Hence, since $L_{1}\times L_{2}$ is a subgroup of $S$, $L_{1}\times L_{2}$ has finite index in $S$.

Summarizing, the key points in Miller's proof are the following ones:\\[5pt]
(1) Subgroups of free groups are free;\\[3pt]
(2) if $S$ is finitely presented, the groups $L_{1}$ and $L_{2}$ are finitely generated;\\[3pt]
(3) finitely generated non-trivial normal subgroups of free groups are of finite index.

In the rest of the section, we show that none of the above conditions necessarily hold for tree groups. For that, we will use the right-angled Artin group $P_{4}$. 

Firstly, let us give an example to show that non-trivial finitely generated normal subgroups do not need to have finite index in tree groups. Consider the homomorphism $\varphi \colon P_{4} \mapsto \mathbb{Z}$ defined as \[ \varphi(a)=\varphi(b)=\varphi(c)=\varphi(d)=1,\] and let $S$ be the kernel of that homomorphism. It can be checked that $S$ is a free group of rank $3$ with basis $\{ ab^{-1},bc^{-1},cd^{-1}\}$. In particular, it is a non-trivial finitely generated normal subgroup of $P_{4}$, but $P_{4} \slash S$ is infinite cyclic.

Secondly, subgroups of tree groups do not need to be tree groups or not even RAAGs. In \cite{Droms}, Droms considers the homomorphism $\alpha \colon P_{4} \mapsto \mathbb{Z}_{2}$ such that \[\alpha(a)=\alpha(b)=\alpha(c)=\alpha(d)=1,\] and shows that the kernel of that homomorphism is not a right-angled Artin group.

Finally, let us give an example of a finitely presented subgroup of $P_{4}\times P_{4}$ such that $L_{1}$ is not finitely generated. Suppose that \[P_{4}^{1}=\langle a,b,c,d\rangle \quad \text{and}  \quad P_{4}^{2}=\langle a^\prime, b^\prime, c^\prime, d^\prime \rangle.\] Consider the homomorphism $f\colon P_{4}^{1}\times P_{4}^{2} \mapsto \mathbb{Z}$ defined as \[f(a)=f(b)=f(d)=1, \quad f(c)=0, \quad f(a^\prime)=f(b^\prime)=f(c^\prime)=f(d^\prime)=1.\]
The kernel of that homomorphism, say $S$, is finitely presented, but $S \cap P_{4}^{1}$ is not finitely generated, see \cite{Bestvina}.

When studying the structure of subgroups of the direct product of $n$ (limit groups over) free groups, Bridson, Howie, Miller and Short require stronger finiteness conditions (see \cite{something1}), namely they consider subgroups of type $FP_n(\mathbb Q)$, in order to generalise Baumslag and Roseblade's result and obtain a structure theorem for these subgroups. In the view of this, one may wonder whether requiring stronger finiteness conditions on the subgroup $S$ may improve the situation, and for instance, ensure that the subgroups $L_{1}$ and $L_{2}$ are finitely generated. However, this is not the case. In our last example, let us consider the short exact sequence
\[ 
\begin{tikzcd}
1 \arrow{r}  & L_{2} \arrow{r} & S\arrow{r}{p_{1}} & P_{4}^{1} \arrow{r} & 1.
\end{tikzcd} 
\]
The associated Hochschild-Lyndon-Serre spectral sequence converging to $H_{\ast}(S;\mathbb{Z})$ has \[E_{pq}^{2}= H_{p}(P_{4}^{1};H_{q}(L_{2};\mathbb{Z})).\]
Since $L_{2}$ is a free group, $H_{q}(L_{2};\mathbb{Z})=0$ for all $q\geq 2$. It follows that the only terms with $p+q=n$ that can possibly be nonzero are \[ H_{n}(P_{4}^{1}; H_{0}(L_{2};\mathbb{Z})) \quad \text{and} \quad H_{n-1}(P_{4}^{1};H_{1}(L_{2};\mathbb{Z})).\]
Right-angled Artin groups are of type $FP_{\infty}$, so both of them are finitely generated abelian groups. Since subgroups and quotients of finitely generated abelian groups are finitely generated, it follows that all the groups on the $E^{\infty}$ page that contribute to $H_{n}(S;\mathbb{Z})$ are finitely generated. Therefore, $H_{n}(S;\mathbb{Z})$ is finitely generated for all $n \geq 0$.

\section{Alternative properties in the class $\mathcal{G}$}
\label{Section 3}

The aim of this section is to study the class $\mathcal{G}$ and to see how the properties of free groups used in Miller's proof generalise for this class. Recall from the previous section that there are three key properties:\\[5pt]
(1) Subgroups of free groups are free;\\[3pt]
(2) if $S$ is finitely presented, the groups $L_{1}$ and $L_{2}$ are finitely generated;\\[3pt]
(3) finitely generated non-trivial normal subgroups of free groups are of finite index.

In this section we will prove that these properties generalise to the following ones for groups in $\mathcal{G}$:\\[5pt]
(1) Subgroups of groups in $\mathcal{G}$ lie in $\mathcal{G}$;\\[3pt]
(2) if $S$ is finitely presented, then a cyclic extension of $L_{i}$ is finitely generated, $i\in \{1,2\}$ (see Proposition~\ref{Proposition 1});\\[3pt]
(3) if $N$ is a non-trivial finitely generated normal subgroup of a group $G\in \mathcal{G}$ with trivial center, then $G \slash N$ is either finite or virtually abelian (see Proposition~\ref{Proposition 0}).

Recall that if $G$ is a graph of groups and $T$ is the Bass-Serre tree corresponding to a splitting of $G$, an element $g$ of a group $G$ satisfies the \emph{weak proper discontinuity} condition (or $g$ is a \emph{WPD element}) if for each vertex group $A$, we have that $A \cap A^g$ is a finite group. In our case, the vertex groups are torsion-free, so the condition that $A \cap A^g$ is a finite group reduces to $A \cap A^g=1$.

In order to consider actions with non-trivial kernel, we define relative WPD elements.

\begin{definition}
Let $G$ be a group in $\mathcal{G}$, let $T$ be the Bass-Serre tree corresponding to a standard splitting of $G$ and let $K$ be the kernel of the action of $G$ on $T$. An element $h$ of $G$ is a \emph{relative WPD element} if for each vertex group $A$, $A \cap A^g$ is the semidirect product of $K$ and a finite group.
\end{definition}

Note that if the vertex groups are torsion-free, then a relative WPD element requires $A \cap A^g= K$. Furthermore, if the action of $G$ on $T$ is faithful, i.e. $K=1$, then a relative WPD element is a WPD element.

We first prove the existence of relative WPD elements in groups in the class $\mathcal{G}$.

\begin{lemma}\label{Lemma WPD}
Let $G$ be a finitely generated group in $\mathcal{G}$, let $T$ be the Bass-Serre tree corresponding to a standard splitting of $G$ and let $K$ be the kernel of the action of $G$ on $T$. Then, $G$ has a relative WPD element. Moreover, there is a finite index subgroup of $G$ that has center $K$.
\end{lemma}

\begin{proof}
Let $G\in \mathcal G$ be finitely generated. If $G$ has a non-trivial free product decomposition, then $G$ has a WPD element, see \cite{Osin}. Therefore, we can assume that $G$ is freely indecomposable, that is, all edge groups are infinite cyclic. 

Let $\Gamma$ be the underlying graph of $G$. Note that since $G$ is finitely generated, $\Gamma$ is a finite graph. Let $Y_0$ be a lift of a maximal tree of $\Gamma$ in the Bass-Serre tree $T$ and let $t_1, \dots, t_s$ be the stable letters corresponding to the edges in $\Gamma \setminus Y_0$. Let $C$ be the intersection $\bigcap_{v\in V(Y_0), g\in \{1, t_1, \dots, t_s\}} {G_v}^g$, where $G_v$ is the vertex stabiliser of $v\in V(Y_0)$.

First, assume that $C$ is infinite cyclic generated by $c$. Then, since edge groups are cyclic, $C$ has finite index in each edge subgroup of a vertex group ${G_v}^g$, $v\in V(Y_0)$, $g\in \{1, t_1, \dots, t_s\}$. The vertex groups are abelian, so $C$ is central in the vertex stabilizers of vertices in $Y_{0}$. Furthermore, since $C$ has finite index in each edge group, we have that for each stable letter $t \in \{t_1, \dots, t_s\}$, there are $n=n(t_i), m=m(t_i)\in \mathbb{Z}$ such that $(c^m)^t=c^n$. Thus, $\langle t, c \rangle$ is isomorphic to the Baumslag-Solitar group $BS(m,n)$. Since $G$ is cyclic subgroup separable, so is $\langle t, c \rangle$ and by \cite[Corollary 7.7]{Levitt}, we have that $|m|=|n|$. Therefore, $t$ normalises the subgroup $\langle c^m \rangle$ and $t^2$ commutes with $c^m$, for each stable letter $t\in \{t_1, \dots, t_s\}$. It follows that there is a power of $c$, say $c^k$, that is normalized by $t_{i}$ and commutes with ${t_{i}}^2$ for all $i\in \{1,\dots,s\}$ and it also commutes with the vertex stabilizers of vertices in $Y_0$. Without loss of generality, we assume that $k$ is positive and minimal with these properties. Therefore, the subgroup $\langle c^k \rangle$ is normal in $G$ and $K= \langle c
^k \rangle$. Since $\langle c^k \rangle$ is the kernel of the action, there is an equivariant epimorphism $G \mapsto G \slash \langle c^k \rangle$. Now, $G \slash \langle c^k \rangle$ acts on $T$ acylindrically since it has finite edge stabilisers. It follows that $G \slash \langle c^k \rangle$ has a WPD element and so $G$ contains a relative WPD element. 

Let $\alpha \colon G \mapsto \langle t_1, \dots, t_s \rangle$ be the epimorphism that sends the vertex groups to $1$ and let $\beta$ be the epimorphism $\langle t_1, \dots, t_s \rangle \mapsto \mathbb Z_2$ that sends each generator $t_i$ to the generator of the group of order 2. The kernel $G^\prime<G$ of the epimorphism $\beta \circ \alpha$ is an index 2 subgroup which contains precisely the set of words $w$ that have an even number of letters $\{t_{1},\dots,t_{s},{t_{1}}^{-1},\dots,{t_{s}}^{-1}\}$. As we mentioned, $(c^k)^{t_i}= c^{\pm k}$ for $i \in \{1,\dots,s\}$ and $c^k$ commutes with the elements of the vertex groups. It follows that if $w \in G^\prime$, then $w$ commutes with $c
^k$. Therefore, $G^\prime$ is an index $2$ subgroup of $G$ with center $\langle c^k \rangle$.

Assume now that $C$ is trivial. Let $g= \prod_{v\in V(Y_0), j\in \{1,\dots, s\}}a_v \cdot {a_v}^{t_j}$, where $a_v$ is an element of the stabiliser $G_v$ of the vertex $v\in V(Y_0)$. 

If $h\in G_v \cap {G_v}^g$, then $h$ fixes the path from the vertex $G_v$ to the the vertex $gG_v$. Since this path contains all the vertices ${G_u}^t$, $u \in V(Y_0)$, $t\in \{1, t_1, \dots, t_s\}$, it follows that $h \in C$ and so $h=1$. Therefore, by \cite[Theorem 4.17]{Osin}, $g$ is a WPD element.
\end{proof}

We now turn our attention to the study of (finitely generated) normal subgroups of groups in the class $\mathcal G$.

\begin{lemma}\label{Lemma 1}
Let $G$ be a finitely generated group in $\mathcal{G}$, let $T$ be the Bass-Serre tree corresponding to a standard splitting of $G$ and let $K$ be the kernel of the action of $G$ on $T$. Suppose that $N$ is a non-trivial normal subgroup of $G$. Then, either $N<K$ or $N$ contains hyperbolic elements and it acts minimally on $T$.
\end{lemma}

\begin{proof}
We first show that if $N$ is elliptic, then it is contained in the kernel $K$ of the action. Indeed, assume that $N$ is a subgroup of a vertex stabilizer $A$. By Lemma~\ref{Lemma WPD}, the group $G$ has a relative WPD element, say $h$, and so we have that $A \cap A^h=K$. Since $N$ is normal, we have that $N^h=N$ is contained in $A \cap A^h = K$ and so $N$ is contained in the kernel of the action.

If $N$ is not elliptic it contains a hyperbolic element. Hence, the union of the axes of such elements is the unique minimal $N$-invariant subtree $X_{0}$ of $T$ (see \cite[Proposition 2.1 (6)]{Bridson2}). Since $N$ is normal in $G$, the $N$-invariant subtree $X_{0}$ is also invariant under the action of $G$. But $T$ is minimal as a $G$-tree, so $X_{0}=T$. Thus, $N$ acts minimally on $T$.
\end{proof}

\begin{proposition}\label{Proposition 0}
Let $G$ be a finitely generated group in $\mathcal{G}$, let $T$ be the Bass-Serre tree corresponding to a standard splitting of $G$, and let $K$ be the kernel of the action of $G$ on $T$. Suppose that $N$ is a non-trivial finitely generated normal subgroup of $G$. Then, either $N<K$ or $G \slash N$ is virtually cyclic. 

Furthermore, if $G\slash N$ is virtually $\mathbb Z$, then $N$ is a free product of free abelian groups whose ranks are bounded above by $r$, where $r$ is the maximum of the ranks of the (free abelian) vertex groups of $G$. In particular, if $G$ is a residually finite tubular group, $N$ is a finitely generated non-trivial normal subgroup of $G$ and $G\slash N$ is virtually $\mathbb Z$, we have that $N$ is a free group.
\end{proposition}

\begin{proof}
Suppose that $N$ is not contained in $K$. By Lemma \ref{Lemma 1}, $N$ acts minimally on $T$. In addition, $N$ is finitely generated, so by \cite[Proposition 7.9]{Bass}, $N \backslash T$ is finite.

Let $C$ be a cyclic edge stabilizer. Then, $|N\backslash G \slash C|$ is finite because the number of edges in $N\backslash T$ is an upper bound for that number.

From the fact that $N$ is normal in $G$, $N \backslash G \slash C$ and $G \slash N C$ are isomorphic. Therefore, $\big{|} G \slash NC \big{|}$ needs to be finite. Notice that if an edge group is trivial, $N$ is of finite index in $G$. Hence, we further assume that $G$ is freely indecomposable, that is, all the edge groups are infinite cyclic.

The Second Isomorphism Theorem gives us that \[NC \slash N \cong C\slash (N \cap C ).\]

It follows that $N \cap C \neq \{1\}$ if and only if $N$ has finite index in $N C$. Therefore, if $N$ intersects non-trivially an edge group $C$, we have that $N$ also has finite index in $G$.

We are left to consider the case when $N$ intersects trivially each edge group in the standard splitting of $G$. In this case, for each edge group $C$, \[ NC \slash N \cong C.\] Thus, since $N C$ has finite index in $G$, $G \slash N$ is virtually $\mathbb Z$. Furthermore, since $N < G$ and $N$ intersects trivially each edge group, $N$ gets induced a decomposition as a free product of free abelian groups, where the free abelian groups are the intersections of the (conjugates of the) vertex groups of $G$ with $N$. Since $G$ is freely indecomposable, each vertex group has an infinite cyclic edge group as subgroup and since $N$ does not intersect any edge group, it follows that the intersection of $N$ with a (conjugate of a) vertex group $G_v$ of $G$ has rank at most the rank $G_v$ minus $1$. In the particular case of tubular groups, since all the vertex groups are isomorphic to $\mathbb Z^2$, we have that the intersection with $N$ is at most of rank $1$ and so $N$ is a free group.
\end{proof}

Our second goal is to find an alternative for the fact that $L_{1}$ and $L_{2}$ are finitely generated in the case of free groups. For the class $\mathcal{G}$, we prove the following:

\begin{proposition}\label{Proposition 1}
Let $G$ be a finitely generated group in $\mathcal{G}$, let $T$ be the Bass-Serre tree corresponding to a standard splitting of $G$ and let $K$ be the kernel of the action of $G$ on $T$. Let $G^\prime$ be any group and let $S < G \times G^\prime$ be a finitely presented subdirect product. Define $L_{1}$ and $L_{2}$ to be $S \cap G$ and $S \cap G^\prime$, respectively. Assume that $L_{1}$ is non-trivial and it is not contained in $K$. Then, there is $y\in G^\prime$ such that $\langle L_{2}, y \rangle$ is finitely generated.
\end{proposition}

\begin{proof}
As $L_{1}$ is non-trivial and it is not contained in $K$, by Lemma~\ref{Lemma 1}, we have that $L_{1}$ contains a hyperbolic isometry, say $t\in L_{1}$ .

Since $G$ is cyclic subgroup separable, it follows from \cite[Theorem 3.1]{Bridson2} that there is a finite index subgroup $M$ in $G$, which is an HNN extension with stable letter $t$ and associated cyclic subgroup $\langle c \rangle$. As $G$ is finitely generated, $M$ is also finitely generated; suppose that $M= \langle s_{1},\dots,s_{n} \rangle,$ with $s_{1}=c$ and $s_{2}=t^{-1}s_{1}t$. 

Let us denote ${p_{1}}_{|S}^{-1}(M)$ by $M^\prime$. Since $M$ has finite index in $G$, $M^\prime$ has finite index in $S$ and since $S$ is finitely presented so is $M^\prime$. The HNN decomposition of $M$ induces a decomposition of $M^\prime$ as an HNN extension. Let us pick $\hat{s_{i}}\in S$ such that $p_{1}(\hat{s_{i}})=s_{i}$ for $i\in \{1,\dots,n\}$. Note that $t$ is an element in $S$. We then have that \[ M^\prime = \langle L_{2},\hat{s_{1}},\dots,\hat{s_{n}}, t \mid t^{-1}\hat{s_{1}}t=\hat{s_{2}}, t^{-1}bt=b, \forall b\in L_{2}, \mathcal{R}^\prime \rangle, \] where $\mathcal{R}^\prime$ is a set of relations in the elements $L_{2} \cup \{\hat{s_{1}},\dots, \hat{s_{n}}\}.$ Recall that $s_{1}=c$, so that $\hat{s_{1}}$ is an element of the form $c y$ in $S$.

Since $M^\prime$ is finitely generated, there are $a_{1},\dots,a_{k}$ in $L_{2}$ such that \[M^\prime= \langle a_{1},\dots,a_{k},\hat{s_{1}},\dots,\hat{s_{n}}, t \mid t^{-1}\hat{s_{1}}t=\hat{s_{2}}, t^{-1}bt=b, \forall b\in L_{2}, \mathcal{R}^\prime \rangle.\]
Let $D$ be the subgroup $\langle a_{1},\dots,a_{k},\hat{s_{1}},\dots,\hat{s_{n}} \rangle$. Note that $L_{2}$ is a subgroup of $D$ because $L_{2}< M^\prime$ and $t$ is an element in $L_{1}$. In conclusion, \[M^\prime= \langle D, t \mid t^{-1}\hat{s_{1}}t=\hat{s_{2}}, t^{-1}bt=b, \forall b\in L_{2}\rangle.\]
Since $M^\prime$ is finitely presented, by \cite[Lemma 2]{Miller}, we deduce that $\langle L_{2}, \hat{s_{1}} \rangle$ is finitely generated. Finally, since $L_{2}$ is contained in $G^\prime$ and $\hat{s_{1}}=c y$ with $c\in G$ and $y\in G^\prime$, we have that the image of $\langle L_{2}, \hat{s_{1}} \rangle$ under the natural projection $\pi \colon G\times G^\prime \mapsto G^\prime$  is finitely generated, that is, $\langle L_{2}, y \rangle$ is finitely generated.
\end{proof}

\begin{remark}\label{rem:finite index}
Let $S$ be a subdirect product of $G_1 \times G_2$ and let us define $L_i$ to be $G_i \cap S$, $i\in \{1,2\}$. Then, \[ G_{1} \slash L_{1} \cong S\slash (L_{1}\times L_{2}) \cong G_{2} \slash L_{2}.\] Indeed, since $S$ surjects onto $G_1$, we have an epimorphism $\pi \colon S \mapsto G_1\slash L_1$ with kernel $L_{1}\times L_{2}$.
\end{remark}

\section{Main result}
\label{Section 4}

The main goal of this section is to prove the following:

\begin{theorem}\label{Main theorem}
Let $S$ be a finitely presented subgroup of $G_{1}\times G_{2}$ where $G_{1}, G_{2} \in \mathcal{A}$. Then, $S$ is virtually $H$-by-(free abelian), where $H$ is the direct product of two groups in $\mathcal{A}$.

Furthermore, $S$ is virtually the kernel of a homomorphism $f$ from $H_1\times H_2$ to $\mathbb R$, for some $H_1, H_2\in \mathcal A$. More precisely, $H$ is equal to $L_1 \times L_2$, where $L_i = S\cap G_i$, $i\in \{1,2\}$ and either
\begin{itemize}
    \item $L_1 \times L_2 <_{fi} S <_{fi} G_1 \times G_2$, or
    \item $S$ is virtually the kernel $\ker f$ where $f\colon H_1 \times H_2 \mapsto \mathbb Z$ for some $H_i \in \mathcal A$, $i\in \{1,2\}$, or
    \item $S$ is virtually the kernel $\ker f$ where $f\colon H_1 \times H_2 \mapsto \mathbb{Z}^2$ for some $H_i\in \mathcal A$, $i\in \{1,2\}$. In this case, $L_i$ is the free product of finitely generated groups in $\mathcal A$ for $i\in \{1,2\}$.
\end{itemize}
\end{theorem}

Let $G_{1}$ and $G_{2}$ be two groups in $\mathcal{A}$ and let $S < G_{1} \times G_{2}$ be a finitely presented subgroup. Let $p_1\colon G_{1}\times G_{2} \mapsto G_{1}$ and $p_2\colon G_{1}\times G_{2} \mapsto G_{2}$ be the two natural projection maps. Then, $S$ is a subdirect product of $p_{1}(S)\times p_{2}(S)$. By Lemma~\ref{Lemma Subgroups}, for $i\in \{1,2\}$, $p_{i}(S)$ is either the direct product of a residually finite free product and a free abelian group (type (1)), or it is the direct product of a group of $\mathcal{G}$ and a free abelian group (type (2)).

\begin{remark} \label{rem:no center}
Suppose that $S$ is a subdirect product of $G_1 \times G_2$ such that $G_i = \mathbb{Z}^{n_{i}} \times H_{i}$, $n_{i}\in \mathbb{N}\cup \{0\}$ and $H_{i}$ is either a residually finite free product or a group in $\mathcal{G}$, for $i\in \{1,2\}$. That is, $S$ is a subdirect product of
\[ (\mathbb{Z}^{n_{1}}\times \mathbb{Z}^{n_{2}}) \times (H_{1}\times H_{2}).\]
Let $\pi$ be the projection map
\[(\mathbb{Z}^{n_{1}}\times \mathbb{Z}^{n_{2}}) \times (H_{1}\times H_{2})\mapsto H_{1}\times H_{2}.\]
Notice that in order to prove Theorem~\ref{Main theorem}, it suffices to prove it for $\pi(S)$. Indeed, if $\pi(S)$ is virtually $H$-by-(free abelian), where $H$ is the direct product of two groups in $\mathcal{A}$, then $S$ is also virtually $H$-by-(free abelian).
\end{remark}

In order to prove Theorem \ref{Main theorem}, we will distinguish three cases depending on the type of $p_{1}(S)$ and $p_{2}(S)$ described in Lemma \ref{Lemma Subgroups}.

\subsection{The groups $p_{1}(S)$ and $p_{2}(S)$ are of type (1)}
\label{Subsection 1}

By Remark \ref{rem:no center}, we can assume that the groups are residually finite free products. 

\begin{theorem}\label{Theorem case 1}
Let $G_{1} \times G_{2}$ be the direct of two finitely generated residually finite groups that admit a non-trivial free product decomposition. Let $S$ be a finitely presented subdirect product in $G_{1}\times G_{2}$ and define $L_{i}$ to be $S \cap G_{i}$, $i\in \{1,2\}$.
\begin{itemize}
\item If $L_{1}=1$, then $S$ is isomorphic to $G_{2}$.
\item If $L_{2}=1$, then $S$ is isomorphic to $G_{1}$.
\item Otherwise, $L_i$ has finite index in $G_i$, $i\in \{1,2\}$ and so $L_{1}\times L_{2}$ has finite index in $S$. If particular, for $i\in \{1,2\}$, if $G_i \in \mathcal A$, then $L_i \in \mathcal A$ and it is finitely generated.
\end{itemize}
\end{theorem}

We will first show that the groups $L_{1}$ and $L_{2}$ defined in the statement are finitely generated.

\begin{proposition}\label{Proposition fg}
Let $A \times K$ be the direct product of a group $A$ with $K$, where $K$ is a finitely generated residually finite group that admits a non-trivial free product decomposition. Suppose that $S$ is a subdirect product in $A \times K$ which intersects $K$ non-trivially. Then, $L= S \cap A$ is finitely generated.
\end{proposition}

\begin{proof}
By hypothesis, since $S$ intersects $K$ non-trivially, there is a non-trivial element $t$ in $S\cap K$. Since $K$ has a non-trivial free product decomposition, it acts minimally on a tree $T$ with trivial edge stabilisers. Moreover, $K$ is residually finite, so the trivial group is closed in the pro-finite topology. Thus, by \cite[Theorem 3.1]{Bridson2}, there is a finite index subgroup $M$ in $K$ which is a free product of the form $B \ast \langle t \rangle$. Since $K$ is finitely generated, so is $M$. Let $\{t,s_{1},\dots,s_{n}\}$ be a generating set for $M$. For $i\in \{1,\dots,n\}$, let  us pick $\hat{s_{i}}\in p_{2}^{-1}(s_{i})$, where $p_{2}$ is the projection map $S \mapsto K$. Let $S_{0}=p_{2}^{-1}(M)$.

Note that $S_{0}$ is of finite index in $S$ since $M$ is of finite index in $K$. The free product decomposition of $M$ induces a splitting of $S_{0}$ of the form: \[S_{0}=\langle L, t, \hat{s_{1}},\dots,\hat{s_{n}} \mid t^{-1}bt=b, \hat{s_{i}}^{-1}b\hat{s_{i}}=\phi_{i}(b), i\in \{1,\dots,n\}, \forall b\in L\rangle,\] where $\phi_{i}$ is the automorphism of $L$ induced by conjugation by $\hat{s_{i}}$.

Now, since $S$ is finitely presented and $S_{0}$ is of finite index in $S$, we have that $S_{0}$ is also finitely presented. Suppose that $S_{0}$ is generated by elements $a_{1},\dots,a_{k}$ in $L$ together with the elements $t,\hat{s}_{1},\dots,\hat{s}_{n}$. Let $D$ be the group $\langle a_{1},\dots,a_{k},\hat{s_{1}},\dots,\hat{s_{n}}\rangle$. Since $L$ is a subgroup of $S_{0}$ and $t\in K$, $L$ is a subgroup of $D$. Moreover, \[S_{0}=\langle D, t \mid t^{-1}bt=b, \forall b\in L \rangle.\]
Finally, since $S_{0}$ is finitely presented, by \cite[Lemma 2]{Miller}, we have that $L$ is finitely generated.
\end{proof}

We now address the proof of Theorem~\ref{Theorem case 1}.

\begin{proof}[Proof of Theorem~\ref{Theorem case 1}]
Note that there are short exact sequences
\[ 
\begin{tikzcd}
1 \arrow{r}  & L_{i} \arrow{r} & S \arrow{r} & G_{j} \arrow{r} & 1,
\end{tikzcd} 
\]
for $i\neq j \in \{1,2\}$, so if $L_{i}$ is trivial for some $i\in \{1,2\}$, $S$ is isomorphic to $G_{j}$ for $j\neq i\in \{1,2\}$.

Now suppose that $L_{1}$ and $L_{2}$ are non-trivial. By Proposition~\ref{Proposition fg}, they are finitely generated. Then, for $i\in \{1,2\}$, $L_{i}$ is a non-trivial finitely generated normal subgroup of $G_{i}$, and since by assumption $G_{i}$ admits a non-trivial free product decomposition, it follows from \cite{Baumslag} that $L_{i}$ has finite index in $G_{i}$. Therefore, $L_{1} \times L_{2}$ has finite index in $S$.
\end{proof}

\subsection{The group $p_{1}(S)$ is of type (1) and $p_{2}(S)$ is of type (2)} 
\label{Subsection 2}

By Remark~\ref{rem:no center}, it suffices to check the result of Theorem~\ref{Main theorem} for $G_{1}\times G_{2}$, where $G_{1}$ is a residually finite free product and $G_{2}$ is a group in $\mathcal{G}$.

\begin{theorem}\label{Theorem case 2}
Let $G_{1}$ be a finitely generated residually finite group that decomposes as a non-trivial free product and let $G_{2}$ be a finitely generated group in $\mathcal{G}$. Let $S$ be a finitely presented subdirect product in $G_{1}\times G_{2}$. Then, $S$ is virtually $H$-by-(free abelian), where $H$ is the direct product of two groups in $\mathcal{A}$.

More precisely, $H$ is equal to $L_1 \times L_2$, where $L_i = S\cap G_i$, $i\in \{1,2\}$ and either
\begin{itemize}
    \item $L_1 \times L_2 <_{fi} S <_{fi} G_1 \times G_2$, or
    \item $S$ is virtually the kernel $\ker f$ where $f\colon H_1 \times H_2 \mapsto \mathbb Z$ for some $H_i <_{fi} G_i$, $i\in \{1,2\}$. Furthermore, if $G_1\in \mathcal A$, then $L_2 < G_2$ is finitely generated and $L_1$ is either finitely generated or the free product of finitely generated groups in $\mathcal A$.
\end{itemize}
\end{theorem}

\begin{proof}
Let us define $L_{1}$ and $L_{2}$ to be $S \cap G_{1}$ and $S\cap G_{2}$, respectively. There are short exact sequences
\[ 
\begin{tikzcd}
1 \arrow{r}  & L_{i} \arrow{r} & S \arrow{r} & G_{j} \arrow{r} & 1,
\end{tikzcd} 
\]
for $i\neq j \in \{1,2\}$, so if $L_{i}$ is trivial for some $i\in \{1,2\}$, $S$ is isomorphic to $G_{j}$ for $j\neq i\in \{1,2\}$.

Now suppose that $L_{1}$ and $L_{2}$ are non-trivial. Let $T$ be the Bass-Serre tree corresponding to a standard splitting of $G_{2}$ and let $K=\langle c \rangle$ be the kernel of the action of $G_{2}$ on $T$. By Lemma~\ref{Lemma WPD} we have that $G_{2}$ has a relative WPD element and that there is a subgroup of finite index, say $K_{2}$, with center $\langle c \rangle$. Then, $S \cap K_{2}$ has finite index in $S$, so it suffices to show that $S\cap K_{2}$ is virtually $H$-by-(free abelian), where $H$ is the direct product of two groups in $\mathcal{A}$. That is, we may assume that $G_{2}= K_{2}$. By Lemma~\ref{Lemma 1}, either $L_{2}< K$ or $L_{2}$ acts minimally on $T$. If $L_{2}< K$, since there is a short exact sequence
\[ 
\begin{tikzcd}
1 \arrow{r}  & L_{2} \arrow{r} & S \arrow{r} & G_{1} \arrow{r} & 1,
\end{tikzcd} 
\]
and $L_{2}$ lies in the center of $K_{2}$, which is cyclic (possibly trivial), we have that $S$ is $L_{2} \times G_{1}$ where $L_2$ is cyclic.

We now need to deal with the case when $L_{2}$ is not contained in $K$. By Proposition~\ref{Proposition fg}, $L_{2}$ is finitely generated. Moreover, there is $e\in G_{1}$ such that $\langle L_{1}, e\rangle$ is finitely generated (see Proposition~\ref{Proposition 1}).

First, by \cite[Theorem 4.1]{something1}, $\langle L_{1}, e\rangle$ has finite index in $G_{1}$. Therefore, $G_{1} \slash L_{1}$ is finite or virtually $\mathbb{Z}$. Second, $L_{2}$ is a finitely generated normal subgroup of $G_{2}$ which is not contained in $K$, so from Proposition~\ref{Proposition 0} we get that $G_{2}\slash L_{2}$ is finite or virtually cyclic. To sum up, since we have that \[ L_{1}\times L_{2} < S < G_{1} \times G_{2},\]
and $(G_{1}\times G_{2}) \slash (L_{1} \times L_{2})$ is virtually abelian, so is $S \slash (L_{1}\times L_{2})$.

Let us now describe the structure of $H=L_1 \times L_2$. Assume that $H$ is not finitely generated. Since $H=L_1\times L_2$ and $L_2$ is finitely generated, it follows that $L_1$ is not finitely generated. Then, we have that $G_1\slash L_1$ is virtually $\mathbb Z$. We claim that in this case $L_1$ is a free product of finitely generated groups from $\mathcal A$. We prove it by induction on the height of the group $G_1$ in the class $\mathcal A$. Since $L_1 < G_1$ and as $G_1$ has, by assumption, a free product decomposition, $L_1$ is a free product of groups of the form $L_1 \cap G_v$ and a free group, where $G_v$ is a vertex stabiliser of the standard splitting of $G_1$. If $G_1\in \mathcal G$, then $G_v$ is a finitely generated free abelian group and so is the intersection with $L_1$. Assume that $G_1$ is of height $k\geq 1$. In this case, $G_v = \mathbb Z^n \times (G_{v_{1}} \ast \dots \ast G_{v_{n}})$ and $G_{v_{i}} \in \mathcal A$ are of height less than $k$, $i\in \{1,\dots,n\}$. Since by the induction hypothesis we have that $L_1 \cap G_{v_{i}}$ is the free product of finitely generated groups from $\mathcal A$, the same holds for $L_1 \cap G_v$.
\end{proof}

\subsection{The groups $p_{1}(S)$ and $p_{2}(S)$ are of type (2).}
\label{Subsection 3}

By Remark~\ref{rem:no center}, we can assume that the groups belong to $\mathcal G$, and by the previous cases, that they are freely indecomposable. Therefore, it suffices to prove the following:

\begin{theorem}\label{Theorem case 3}
Let $G_{i}$ be a freely indecomposable finitely generated group in $\mathcal{G}$, for $i\in \{1,2\}$. Let $S$ be a finitely presented subdirect product of $G_{1}\times G_{2}$. Then, $S$ is virtually $H$-by-(free abelian), where $H$ is the direct product of two groups in $\mathcal{G}$.

More precisely, $H$ is equal to $L_1 \times L_2$, where $L_i = S\cap G_i$, $i\in \{1,2\}$ and either
\begin{itemize}
    \item $L_1 \times L_2 <_{fi} S <_{fi} G_1 \times G_2$, or
    \item $S$ is virtually the kernel $\ker f$ where $f\colon H_1 \times H_2 \mapsto \mathbb Z$ for some $H_i \in \mathcal A$, $i\in \{1,2\}$, or
    \item $S$ is virtually the kernel $\ker f$ where $f\colon H_1 \times H_2 \mapsto \mathbb{Z}^2$ for some $H_i\in \mathcal A$, $i\in \{1,2\}$. In this case, $L_i$ is the free product of finitely generated groups in $\mathcal A$ for $i\in \{1,2\}$.
\end{itemize}
\end{theorem}

\begin{proof}
Let $p_1\colon S \mapsto G_{1}$ and $p_2\colon S\mapsto G_{2}$ be the natural projection maps and let $L_1=S \cap G_{1}$ and $L_2=S \cap G_{2}$. Let us deal with the case when $L_{1}$ or $L_{2}$ is trivial. There are short exact sequences
\[ 
\begin{tikzcd}
1 \arrow{r}  & L_{i} \arrow{r} & S \arrow{r} & G_{j} \arrow{r} & 1,
\end{tikzcd} 
\]
for $i\neq j \in \{1,2\}$, so if $L_{i}$ is trivial for some $i\in \{1,2\}$, $S$ is isomorphic to $G_{j}$ for $j\neq i\in \{1,2\}$.

Now, suppose that $L_{1}$ and $L_{2}$ are non-trivial. For $i\in \{1,2\}$, let $T_{i}$ be the Bass-Serre tree corresponding to a standard splitting of $G_{i}$ and let $K_{i}= \langle c_{i} \rangle$ be the kernel of the action of $G_{i}$ on $T_{i}$. By Lemma~\ref{Lemma 1}, there is a finite index subgroup, say $H_{i}$, in $G_{i}$ with center $\langle {c_{i}} \rangle$ and $G_{i}$ has a relative WPD element. Since $H_{i}$ has finite index in $G_i$, we may assume that $G_{i}= H_{i}$. 

By Lemma~\ref{Lemma 1}, either $L_{i}< K_{i}$ or $L_{i}$ acts minimally on $T_{i}$. If $L_{i} <K_{i}$ for some $i\in \{1,2\}$, say $L_1 < K_1$, then $L_{1}$ is contained in the center of the group $H_{1}$. Therefore, $S$ is virtually $L_{1} \times G_{2}$, where $L_1$ is cyclic.

We now deal with the case when $L_{i}$ is not contained in $K_{i}$, for $i\in \{1,2\}$. In this case, we show that $G_{2}\slash L_{2}$ is finite or virtually free abelian. A symmetric argument shows that $G_{1} \slash L_{1}$ is finite or virtually free abelian. Therefore, since \[L_{1}\times L_{2} < S < G_{1} \times G_{2},\] $S \slash (L_{1} \times L_{2})$ is virtually free abelian.

Let $T= T_{2}$ be the Bass-Serre tree associated to a standard splitting of $G_2$. By Lemma~\ref{Lemma 1}, $L_{2}$ acts minimally on $T$. By Proposition~\ref{Proposition 1}, there is $e\in G_{2}$ such that $\langle L_{2}, e \rangle$ is finitely generated. This group contains $L_{2}$, so it also acts minimally on $T$. Furthermore, it is finitely generated, so by \cite[Proposition 7.9]{Bass}, the graph $\langle L_{2},e\rangle \backslash T$ is finite. Then, for every cyclic edge stabilizer $\Gamma_{e}$, we have that \[\big{|} \langle L_{2}, e \rangle \backslash G_{2} \slash \Gamma_{e} \big{|} \] is finite since it is bounded above by the number of edges in the graph $\langle L_{2},e\rangle \backslash T$. Then, there are $z_{1},\dots,z_{m}\in G_{2}$ such that 

\begin{equation}\label{eq:decomposition}
 G_{2}= \dot{\bigcup}_{j\in \{1,\dots,m\}} \langle  L_{2},e \rangle z_{j}\Gamma_{e}.
\end{equation}

Suppose that $C= \langle c \rangle$ is a cyclic edge stabilizer and $C < A$ is a free abelian vertex stabilizer. 

Assume first that there is $a\in A < G_2$ such that $c^n \neq a^m$ in $G_{2}\slash L_{2}$ for all $n,m \in \mathbb{Z}$.

Since the set $\{z_{1},\dots,z_{m}\}$ is finite, there are two different powers of $a$ that are in the same double coset, that is, there exist $n_{1}\neq n_{2}\in \mathbb{N}$ and $j\in \{1,\dots,m\}$ such that
\[ a^{n_{1}}= l_{2}e^{k_{1}}z_{j}c^{m_{1}}\quad \text{and} \quad a^{n_{2}}= {l_{2}}^\prime e^{k_{2}}z_{j} c^{m_{2}},\] for some $l_{2},{l_{2}}^\prime \in L_{2}$, $k_{1},m_{1},k_{2},m_{2}\in \mathbb{Z}$.

Equating $z_{j}$ in the previous equations and using the fact that $L_{2}$ is normal in $G_{2}$, we deduce that there is $l\in L_{2}$ such that
\[ l e^{k_{2}-k_{1}}= a^{n_{2}}c^{m_{1}-m_{2}}a^{-n_{1}}.\]
Note that, by the standing assumption $c^n\ne a^m$ in $G_2\slash L_2$, we have that $k_{1}\neq k_{2}$ (and so without loss of generality, we assume that $k_{2}> k_{1}$). Indeed, otherwise, $a^{n_{1}-n_{2}}$ would be equal to $c^{m_{1}-m_{2}}$ modulo $L_{2}$. Since $a,c \in A$, we deduce that $e^{k_{2}-k_{1}}$ is congruent to an element of $A$, say $a^\prime$, modulo $L_{2}$. If we define $z_{s,j}$ to be $e^s z_{j}$, for $s\in \{0,\dots,k_{2}-k_{1}-1\}$ and $j\in \{1,\dots,m\}$, we have that
\[ G_{2}= \dot{\bigcup_{j\in \{1,\dots,m\}, s\in \{0,\dots,k_{2}-k_{1}-1\}}}\big{\langle} L_{2} ,e^{k_{2}-k_{1}}\big{\rangle} z_{s,j} \langle c \rangle =\]
\[\dot{\bigcup_{j\in \{1,\dots,m\}, s\in \{0,\dots,k_{2}-k_{1}-1\}}}\big{\langle} L_{2}, a^\prime  \big{\rangle} z_{s,j} \langle c \rangle.\]

The next goal is to obtain a decomposition of $G_{2}$ as a disjoint union of single cosets.

The element $c^t z_{s,j}$ lies in $G_{2}$ for all $t\in \mathbb{Z}$, $j\in \{1,\dots,m\}, s\in \{0,\dots,k_{2}-k_{1}-1\}.$ Therefore, there are distinct natural numbers $t_{1}$ and $t_{2}$ and $s_{0}\in \{0,\dots,k_{2}-k_{1}-1\}$, $j_{0} \in \{1,\dots,m\}$ such that
\[ c^{t_{1}} z_{s,j}=l (a^\prime)^{s_{1}} z_{s_{0},j_{0}} c^{m_{1}} \quad \text{and} \quad \] \[ c^{t_{2}} z_{s,j}=l^\prime (a^\prime)^{s_{2}} z_{s_{0},j_{0}} c^{m_{2}}, \]for some $l, l^\prime \in L_{2}$, $s_{1},m_{1},s_{2},m_{2}\in \mathbb{Z}$.

Equating the $z_{s_0, j_0}$ and using the normality of $L_2$, we deduce that there is $l^{\prime \prime}\in L_{2}$ such that
\[ l^{\prime \prime}c^{-t_{2}} (a^\prime)^{s_{2}-s_{1}} c^{t_{1}} z_{s,j}= z_{s,j} c^{m_{1}-m_{2}}.\]
Denote the element $c^{-t_{2}} (a^\prime)^{s_{2}-s_{1}} c^{t_{1}}\in A$ by $a^{\prime \prime}\in A$. Then the previous equation is of the form
\[
l^{\prime \prime} a^{\prime \prime} z_{s,j}= z_{s,j}c^{m_1-m_2},
\]
where $l^{\prime \prime}\in L_2$ and $a^{\prime \prime}\in A$. Again, by the standing assumption, we have that $m_{1}$ and $m_{2}$ are different and by taking further cosets as we did in the previous case, there are $f_{1},\dots,f_{r}\in G_{2}$ such that
\[ G_{2}= \dot{\bigcup}_{j\in \{1,\dots,r\}} \langle  L_{2},a^{\prime}, a^{\prime \prime} \rangle f_{j}.\]

Summarizing, the subgroup $\langle L_{2}, a^{\prime}, a^{\prime \prime} \rangle $ has finite index in $G_{2}$ and so $L_{2}A$ has finite index in $G_{2}$. Moreover, by the Second Isomorphism Theorem, $L_{2}A \slash L_{2}$ is isomorphic to $A \slash (A \cap L_{2})$, and since $A$ is abelian, so is $L_{2}A \slash L_{2}$. Therefore, under the standing assumption, we have that $G_{2} \slash L_{2}$ is virtually free abelian.

We now deal with the case when for each vertex stabilizer $A$ and each edge stabilizer $C=\langle c \rangle < A$, for each $a\in A$ there are $n,m\in \mathbb{Z}$ such that 

\begin{equation}\label{eq:first}
    a^n=c^m \quad \text{in} \quad G_{2} \slash L_{2}.
\end{equation}

In particular, for any two edge stabilizers $\langle c_{1} \rangle$ and $\langle c_{2} \rangle$, there are $w_{1},w_{2}\in \mathbb{Z}$ such that ${c_{1}}^{w_{1}}= {c_{2}}^{w_{2}}$ modulo $L_{2}$.

Recall the double coset splitting of $G_{2}$,

\begin{equation}\label{eq:second}
G_{2}= \dot{\bigcup}_{j\in \{1,\dots,m\}} \langle  L_{2},e \rangle z_{j} \Gamma_{e}. 
\end{equation}

Suppose that $\Gamma_{e}= \langle \gamma_{e} \rangle$. Note that for each $j\in \{1,\dots,m\}$, \[\langle L_{2}, e \rangle z_{j} \langle \gamma_{e} \rangle = \langle L_{2}, e \rangle z_{j} \langle \gamma_{e} \rangle z_{j}^{-1}z_{j}= \langle L_{2}, e \rangle \langle {\gamma_{e}}^{{z_{j}}^{-1}} \rangle z_{j}.\]
Observe that $\langle {\gamma_{e}}^{{z_{j}}^{-1}} \rangle$ is an edge stabilizer, so by assumption, there are two numbers $n_{j},m_{j}\in \mathbb{Z}$ such that \[ \big{(}{\gamma_{e}}^{{z_{j}}^{-1}}\big{)}^{n_{j}}= {\gamma_{e}}^{m_{j}} \quad \text{in} \quad G_{2}\slash L_{2}.\]
We again take cosets as in the previous cases to get that

\begin{equation}\label{eq:third}
    G_{2}= \dot{\bigcup}_{j\in \{1,\dots,m\}} \langle  L_{2},e, \gamma_{e} \rangle q_{j},
\end{equation}
for some $q_{1},\dots,q_{m}\in G_{2}$.

Let us distinguish two cases. First, suppose that $e$ is elliptic. Then, $e\in A$ for a vertex stabilizer $A$. Let $\langle c \rangle$ be an edge stabilizer such that $c\in A$. We are now in the same situation as in the previous case since $e,c\in A$: $L_{2}A$ has finite index in $G_{2}$ and so we have again that $G_{2} \slash L_{2}$ is virtually free abelian.

Finally, we need to deal with the case when $e$ is hyperbolic. 

By \cite[Corollary 3.2]{Bridson2}, there is $M$ a finite index subgroup of $G_{2}$ such that $M$ is an HNN extension with stable letter $e$ and amalgamated subgroup $M \cap C$, where $C$ is an edge stabilizer of $T$. If $c$ is a generator of $C$, then $M\cap C= \langle c^r \rangle$ for some $r\in \mathbb{N} \cup \{0\}$. Let us denote $c^r$ by $c_{1}$. Since $M$ has finite index in $G_{2}$, $M$ is finitely presented and admits a presentation of the form \[ \langle b_{1},\dots,b_{k},c_{1},e,c_{2} \mid \mathcal{R}, e^{-1}c_{1}e=c_{2}\rangle,\] where $\mathcal{R}$ is a set of relations in the words $\{b_{1},\dots,b_{k},c_{1},c_{2}\}$ and $c_{2}$ is a power of a generator of another edge stabilizer in $T$. Then, we are in the following situation:
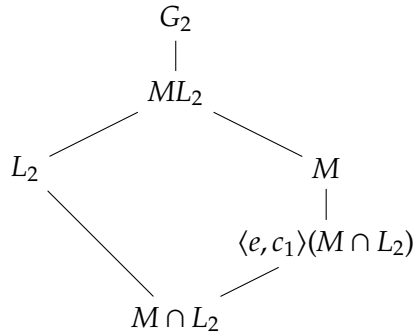
\[\begin{tikzpicture}
  \node (max) at (0,0) {$G_2$};
  \node (a) at (0,-1) {$ML_{2}$};
  \node (b) at (-2,-2) {$L_{2}$};
  \node (c) at (2,-2) {$M$};
  \node (d) at (2,-3) {$\langle e,c_{1} \rangle (M \cap L_{2})$};
   \node (f) at (0,-4) {$M \cap L_{2}$};
  \draw (max) -- (a) -- (b)
  (a)--(c)
  (c)--(d)
  (b) -- (f)
  (d)--(f);
\end{tikzpicture}
\]

From the double coset representation \eqref{eq:third} applied to $c_1$, we have that
\begin{equation}\label{eq:four}
    G_{2}= \dot{\bigcup}_{j\in \{1,\dots,m\}} \langle  L_{2},e, c_{1} \rangle q_{j}.
\end{equation}
The group $L_{2}$ is a subgroup of $ML_{2}$ and the elements $e,c_{1}$ belong to $M< ML_{2}$. Then, there are $z_{1},\dots, z_{t}\in ML_{2}$ such that
\[ ML_{2}= \dot{\bigcup}_{j\in \{1,\dots,t\}} \langle  L_{2},e, c_{1} \rangle z_{j}.\] Thus, there are $\overline{z_{1}},\dots, \overline{z_{t}} \in M$ such that
\[ M= \dot{\bigcup}_{j\in \{1,\dots,t\}} \langle  M \cap L_{2},e,c_{1} \rangle \overline{z_{j}}.\]
Assume that the elements $\overline{z_{1}},\dots, \overline{z_{s}}$ belong to $\langle e, c_{1}\rangle (M \cap L_{2})$ and that $\overline{z_{s+1}},\dots, \overline{z_{t}}$ do not belong to $\langle e,c_{1} \rangle (M \cap L_{2})$. Then, we have the double coset decomposition of $\langle e, c_{1}\rangle (M \cap L_{2})$:
\begin{equation}\label{eq:five}
    \langle e, c_{1} \rangle (M \cap L_{2})= \dot{\bigcup}_{j\in \{1,\dots,s\}} \langle  M \cap L_{2},e,c_{1} \rangle \overline{z_{j}}.
\end{equation}

By the standing assumption \eqref{eq:first}, we have that $e^{-1}c_{1}^m e= c_{1}^n$ modulo $L_{2}$ for some $m,n\in \mathbb{Z}$. Therefore, the group $\langle e, c_{1} \rangle (M \cap L_{2}) \slash (M \cap L_{2})$ is isomorphic to the quotient of the Baumslag-Solitar group \[BS(m,n) = \langle x, t \mid t^{-1}x^m t =x^n\rangle.\] That is, there is $N$ a normal subgroup of $BS(m,n)$ and an isomorphism \[f \colon \langle e, c_{1} \rangle (M \cap L_{2}) \slash (M \cap L_{2}) \mapsto BS(m,n) \slash N\] with $f(e M \cap L_{2})= tN, f( c_{1} M \cap L_{2})= x N.$

By the decomposition given in \eqref{eq:five}, there are elements $a_{1},\dots,a_{s}\in BS(m,n)$ such that
\begin{equation} \label{eq:six}
    BS(m,n)= \dot{\bigcup}_{j\in \{1,\dots,s\}} N \langle t, x \rangle a_{j}.
\end{equation}
If $m$ is equal to $n$, ${c_{1}}^m$ commutes with $e$ modulo $L_{2}$. Then, from the decomposition given in \eqref{eq:four}, we have that $G_{2} \slash L_{2}$ is virtually free abelian. If $m$ equals $-n$, ${c_{1}}^m$ commutes with $e^2$ modulo $L_{2}$. From the decomposition \eqref{eq:four} we get that there are elements $r_{1},\dots,r_{g}\in G_{2}$ such that
\[ G_{2}= \dot{\bigcup}_{j\in \{1,\dots,g\}} \langle  L_{2},e^2, {c_{1}}^m \rangle r_{j},\] so $G_{2} \slash L_{2}$ is virtually abelian.

Let us deal with the case when $m$ is not equal to $\pm n$. In this case, our aim is to show that there is $q\in \mathbb{N}\setminus \{0\}$ such that $x^q \in N$ and using the isomorphism $f$, we deduce that ${c_{1}}^q \in M \cap L_{2} < L_{2}$. Notice that if ${c_{1}}^q \in L_2$ for $q\ne 0$, then it follows from the decomposition \eqref{eq:four} that  $G_2 \slash L_2$ is virtually cyclic.

Note that we can assume that $\gcd(m,n)=1$. Otherwise, if $\gcd (m,n)=d$, then $m= d m^\prime$ and $n= dn^\prime$ for some $m^\prime, n^\prime \in \mathbb{Z}$ with $\gcd (m^\prime, n^\prime)=1$ and $e^{-1}({c_{1}}^d)^{m^\prime} e= ({c_{1}}^d)^{n^\prime}$ modulo $L_{2}$.

Let us denote the normal closure of $x$ in $BS(m,n)$ by $\langle \langle x \rangle \rangle$. Let us prove that for each $g\in \langle \langle x \rangle \rangle$ there is $S = S(g)\in \mathbb{N}$ such that $(x^S)^g= x^S.$ If $g\in \langle \langle x \rangle \rangle$, then \[g= (g_{0} x^{\pm 1} g_{0}^{-1})(g_{1} x^{\pm 1} g_{1}^{-1})\dots (g_{n} x^{\pm 1} g_{n}^{-1}),\] for some $g_{i}\in BS(m,n)$, $i\in \{0,\dots,n\}$. For each $i\in \{0,\dots,n\}$, let \[x_{i}= \max \big \{ \text{number of } t\text{'s in } g_{i}, \text{number of } t^{-1}\text{'s in } g_{i}\big \},\] and let $S= \max \big \{x_{i} \mid i\in \{1,\dots, n\} \big \}$. Then,
\[{\left( x^{|m|^S |n|^S}\right) }^g= x^{|m|^S |n|^S}.\]

We now consider two cases based on whether or not $N$ is contained in $\langle \langle x \rangle \rangle$.

Suppose that $N$ is not contained in $\langle \langle x \rangle \rangle$, that is, there is an element $h\in N$ which is not in $\langle \langle x \rangle \rangle$. Since $BS(m,n)= \langle \langle x \rangle \rangle \langle t \rangle$, we can write $h= g t^k$ for some $g\in \langle \langle x \rangle \rangle$ and $k\in \mathbb{Z}\setminus \{0\}$. By the previous paragraph, there is $S$ such that ${\left( x^{|m|^S |n|^S}\right) }^g= x^{|m|^S |n|^S}$. Let us take $M$ to be $\max \{S, |k| \}$.

If $m$ and $n$ are greater than $0$, then
\[{\left( x^{m^M n^M}\right)}^h= x^{m^{M-k} n^k n^M} \text{ if } k>0, \quad {\left( x^{m^M n^M}\right)}^h= x^{n^{M-k} m^k m^M} \text{ if } k<0.\] So in $BS(m,n) \slash N$, \[x^{m^M n^M} N= x^{m^{M-k} n^k n^M} N \quad \text{or} \quad x^{m^M n^M} N= x^{n^{M-k} m^k m^M} N.\] That is,
\[x^{n^M m^{M-k}(m^k-n^k)}\in N \quad \text{or} \quad x^{m^M n^{M-k}(n^k-m^k)}\in N.\]
Since $m$ is not equal to $\pm n$, $m^k-n^k\neq 0$ and $n^k-m^k\neq 0$. So there is $q\in \mathbb{N}$ such that $x^q\in N$.

If $m<0$ and $n>0$, \[{(x^{|m|^M n^M})}^h= (x^{-1})^{|m|^{M-k} n^k n^M} \text{ if } k>0, \quad {(x^{|m|^M n^M})}^h= (x^{-1})^{n^{M-k} |m|^k |m|^M} \text{ if } k<0.\] Therefore, as in the previous case,
\[x^{n^M |m|^{M-k}(|m|^k-n^k)}\in N \quad \text{or} \quad x^{|m|^M n^{M-k}(n^k-|m|^k)}\in N.\] Since $m$ is not equal to $\pm n$, again we have that there is $q\in \mathbb{N}$ such that $x^q \in N$.

Therefore, we are left to consider the case when $N$ is a subgroup of $\langle \langle x \rangle \rangle$. Let us first show that $\langle \langle x \rangle \rangle \slash N$ is virtually cyclic. For that, we show that $\langle N , x \rangle$ has finite index in $\langle \langle x \rangle \rangle$. Let $g\in \langle \langle x \rangle \rangle$. By the decomposition \eqref{eq:six}, there is an element $n\in N$, some $m,k\in \mathbb{Z}$ and $j\in \{1,\dots,s\}$ such that $g= n t^m a_{j} x^k $. Observe that $n^{-1}gx^{-k}$ is an element of $\langle \langle x \rangle \rangle$ and so $t^m a_{j}$ also belongs to $\langle \langle x \rangle \rangle$. The sum of the powers of $t$ is $0$ in every element of $\langle \langle x \rangle \rangle$, so since $a_{j}$ is a fixed element, $m$ also needs to be a fixed number, say $k_j$. Therefore, from the above observation and the decomposition \eqref{eq:six} we get that
\[ \langle \langle x \rangle \rangle= \dot{\bigcup} N y_{i} \langle x \rangle ,\]
where $y_{i}= t^{k_{i}}a_{i} \in \langle \langle x \rangle \rangle$ and so $\langle \langle x \rangle \rangle \slash N$ is virtually cyclic.

Let us denote $\langle \langle x \rangle \rangle$ by $H$ and its commutator subgroup by $[H,H]$. Then, we are in the following situation:
\[\begin{tikzpicture}
  \node (max) at (0,0) {$\langle \langle x \rangle \rangle= H$};
  \node (a) at (0,-1) {$[H,H]\langle N, x \rangle$};
  \node (b) at (-2,-2) {$\langle N,x \rangle$};
  \node (d) at (-2,-3) {$N$};
  \node (e) at (3,-3) {$[H,H]$};
   \node (f) at (0,-2) {$[H,H] N$};
  \draw (max) -- (a) -- (b) -- (d)
  (a) -- (e)
  (f)--(d)
  (f)--(e)
  (a)--(f);
\end{tikzpicture}
\]
Since $H \slash N$ is virtually cyclic, then either $N$ has finite index in $\langle N, x \rangle$ or $\langle N, x \rangle \slash N$ is cyclic. In the former, we have that $x^q \in N$. If $\langle N, x \rangle \slash N$ is cyclic, we have that $[H,H] \langle N,x \rangle \slash [H,H]N$ is also cyclic and so $N$ has finite index in $[H,H]N$. 

By Lemma~\ref{Lemma BS}, we have that $H \slash [H,H]$ is isomorphic to $\mathbb{Z}\left[\frac{1}{mn}\right]$. The group $[H,H] N \slash [H,H]$ is a $\mathbb{Z}$-submodule of $\mathbb{Z}[\frac{1}{mn}]$. If $[H,H] N \slash [H,H]$ is trivial, then $N \subseteq [H,H]$ and as $H \slash [H,H]\cong \mathbb{Z}\left[\frac{1}{mn}\right]$ and $H\slash N$ is virtually cyclic, it follows that $mn= \pm 1$. This case is covered in the case when $m$ is equal to $\pm n$. Therefore, we need to deal with the case when $[H,H]N \slash [H,H]$ is non-trivial. In this case, there is $y$ a non-trivial element in $[H,H] N \slash [H,H]$, so $y$ is of the form $\frac{d}{(mn)^k}$ for some $d\in \mathbb{Z}\setminus \{0\}$ and $k\in \mathbb{N}$. But since $[H,H] N \slash [H,H]$ is a $\mathbb{Z}$-submodule, then $d\in [H,H] N \slash [H,H]$. Therefore, $x^d\in [H,H] N$. Since $N$ has finite index in $[H,H] N$, we have thast $x^{q}\in N$ for some $q\in \mathbb{N}$. Therefore, in all cases $G_2\slash L_2$ is virtually abelian.

Note that the kernel of the natural epimorphism $ f_{i} \colon S \mapsto G_{i} \slash L_{i}$ is $L_{1}\times L_{2}$, so \[ G_{1} \slash L_{1} \cong S \slash (L_{1} \times L_{2}) \cong G_{2} \slash L_{2}.\]
We have just proved that $G_{i} \slash L_{i}$ is either finite, virtually $\mathbb{Z}$ or virtually $\mathbb{Z}^2$. If $G_{i} \slash L_{i}$ is finite, since $G_{i}$ is finitely generated, so is $L_{i}$. 

If $G_{i}\slash L_{i}$ is virtually $\mathbb{Z}$, there is $H$ a finite index subgroup in $S$ such that $H \slash (L_{1} \times L_{2})$ is $\mathbb{Z}$ and $(f_{1}(H) \times f_{2}(H))\slash (L_{1} \times L_{2})$ is $\mathbb{Z}^2$. Thus $H$ is normal in $f_{1}(H) \times f_{2}(H)$ and $(f_{1}(H) \times f_{2}(H)) \slash H$ is $\mathbb{Z}$. Since $f_{i}(H)$ is a finite index subgroup of $G_{i}$, it lies in $\mathcal{G}$ and it is finitely generated. In conclusion, $H$ is a finite index subgroup of $S$ and it is the kernel of a homomorphism $f_{1}(H) \times f_{2}(H) \mapsto \mathbb{Z}$, where $f_{i}(H) \in \mathcal{G}$.

Finally, we deal with the case when $G_{i} \slash L_{i}$ is virtually $\mathbb{Z}^2$. Since the group $L_i$ is a subgroup of $G_i$, it acts on the Bass-Serre tree $T_i$ and so it inherits a graph of groups decomposition. We claim that the intersection of $L_{i}$ with each edge group is trivial and so the decomposition of $L_i$ is in fact a non-trivial free product decomposition. Let us prove it for $i=2$ being the case $i=1$ analogous. If $\Gamma_{e}$ is a cyclic edge stabilizer of $T_{2}$ such that $L_{2} \cap \Gamma_{e} \neq 1$, then from \eqref{eq:decomposition} we obtain that $G_{2} \slash L_{2}$ is virtually $\mathbb{Z}$ contradicting our assumption.
\end{proof}

The following lemma is a well-known fact on Baumslag-Solitar groups but we add the proof here for completeness.

\begin{lemma}\label{Lemma BS}
Let $BS(m,n)= \langle x,t \mid t^{-1}x^m t= x^n \rangle$ such that $\gcd(m,n)=1$ and denote the normal closure of $x$ in $BS(m,n)$ by $\langle \langle x \rangle \rangle$. Then, $H_{1}(\langle \langle x \rangle \rangle; \mathbb{Z})$ is isomorphic to $\mathbb{Z}\left[\frac{1}{mn}\right]$.
\end{lemma}

\begin{proof}
Let us denote $\langle \langle x \rangle \rangle$ by $L$. If $m\in \{1,-1\}$ or $n\in \{1,-1\}$, $L$ is free abelian and isomorphic to $\mathbb{Z}\left[\frac{1}{n}\right]$ (see \cite{Collins}). Thus, we may assume that $m,n\in \mathbb{Z} \setminus \{0,1,-1\}$.

Let $x_{i}=t^{-i}x t^{i}$ for $i\in \mathbb{Z}$. Then, $x_{i+1}^m= x_{i}^n$ and $L$ has a decomposition as a two-way infinite amalgamated free product:
\[ \cdots \ast_{\langle x_{-1} \rangle} \langle x_{-1}, x_{0} \mid x_{-1}^n= x_{0}^m \rangle \ast_{\langle x_{0} \rangle} \langle x_{0}, x_{1} \mid x_{0}^n= x_{1}^m \rangle  \ast_{\langle x_{1} \rangle} \cdots \]
Let us define the epimorphism $f \colon L \mapsto \mathbb{Z}\left[\frac{m}{n}, \frac{n}{m}\right]$ such that $f(i)= \left(\frac{m}{n}\right)^{i}$ for all $i\in \mathbb{Z}$. That is, \[ f(x_{0})= 1, \quad f\left(x_{i}\right)= \left(\frac{m}{n}\right)^{i} \text{ if } i>0, f(x_{i})= \left(\frac{n}{m}\right)^{-i} \text{ if } i<0.\]
Let us first show that $\mathbb{Z}\left[\frac{m}{n}, \frac{n}{m}\right]$ and $\mathbb{Z}\left[\frac{1}{mn}\right]$ are isomorphic. Since the greater common divisor of $m$ and $n$ is $1$, there are $k_{1}, k_{2} \in \mathbb{Z}$ such that $1= mk_{1} + nk_{2}$. Then, \[ \frac{1}{m}= k_{1}+ \frac{n}{m}k_{2} \quad \text{and} \quad \frac{1}{n}= k_{2}+ \frac{m}{n}k_{1}.\] Therefore, $\mathbb{Z}\left[\frac{m}{n}, \frac{n}{m}\right] = \mathbb{Z}\left[\frac{1}{n},\frac{1}{m}\right]$ and this is clearly isomorphic to $\mathbb{Z}[\frac{1}{mn}]$.

Finally, we need to check that $\ker f$ coincides with the commutator subgroup $L^\prime= [L,L]$. Since $L \slash \ker f$ is abelian, $L^\prime$ is contained in $\ker f$, so we need to prove that $\ker f \subseteq L^\prime$.

Firstly, we show that if $x_{i_{1}}^{\pm 1} \dots x_{i_{k}}^{\pm 1} \in \ker f$, the number of $x_{i_{j}}$'s in the word $x_{i_{1}}^{\pm 1} \dots x_{i_{k}}^{\pm 1}$ is equal to the number of $x_{i_{j}}^{-1}$'s in $x_{i_{1}}^{\pm 1} \dots x_{i_{k}}^{\pm 1}$, for each $j\in \{1,\dots,k\}$. Let us prove it by induction on $k$. If $k$ is equal to $1$, $x_{1}^{\pm 1}$ is not an element in $\ker f$. If $k=2$, there are some options for $f(x_{i_{1}}^{\pm 1} x_{i_{2}}^{\pm 1})$:\\[5pt]
(1) $f(x_{i_{1}}^{\pm 1} x_{i_{2}}^{\pm 1})= \pm (\frac{m}{n})^{i_{1}} \pm (\frac{m}{n})^{i_{2}},$\\[3pt]
(2) $f(x_{i_{1}}^{\pm 1} x_{i_{2}}^{\pm 1})= \pm (\frac{m}{n})^{i_{1}} \pm 1,$\\[3pt]
(3) $f(x_{i_{1}}^{\pm 1} x_{i_{2}}^{\pm 1})= \pm (\frac{m}{n})^{i_{1}} \pm (\frac{n}{m})^{i_{2}},$\\[3pt]
(4) $f(x_{i_{1}}^{\pm 1} x_{i_{2}}^{\pm 1})= \pm 1 \pm 1,$\\[3pt]
(5) $f(x_{i_{1}}^{\pm 1} x_{i_{2}}^{\pm 1})= \pm (\frac{n}{m})^{i_{1}} \pm (\frac{n}{m})^{i_{2}},$\\[3pt]
(6) $f(x_{i_{1}}^{\pm 1} x_{i_{2}}^{\pm 1})= \pm (\frac{n}{m})^{i_{1}} \pm 1.$\\[5pt]

In the first case, $f(x_{i_{1}}^{\pm 1} x_{i_{2}}^{\pm 1})=0$ if and only if $i_{1}= i_{2}$ and the powers have opposite sign. The cases (4) and (5) are similar. In cases (2), (3) and (6), $f(x_{i_{1}}^{\pm 1} x_{i_{2}}^{\pm 1})$ is not $0$.

Now suppose that the statement holds for $k-1$ and let us check it for $k$. Let $x_{i_{1}}^{\pm 1} \dots x_{i_{k}}^{\pm 1} \in \ker f$ and suppose that $i_{l_{1}},\dots, i_{l_{n}}$ are the positive values among $i_{1},\dots,i_{k}$ and $i_{t_{1}},\dots, i_{t_{s}}$ are the negative ones. Then,
\[ \pm \left(\frac{m}{n}\right)^{i_{l_{1}}} \pm \dots \pm \left(\frac{m}{n}\right)^{i_{l_{n}}} \pm \left(\frac{n}{m}\right)^{-i_{t_{1}}} \pm \cdots \pm \left(\frac{n}{m}\right)^{-i_{t_{s}}} + k=0.\]
By taking $n^{i_{l_{1}}+\dots + i_{l_{n}}}m^{-(i_{t_{1}}+\dots + i_{t_{s}})}$ as the denominator, we have
\[\pm n^{j_{1}}m^{l_{1}}\pm n^{j_{2}}m^{l_{2}} \pm \dots \pm n^{j_{k}}m^{l_{k}}=0,\] for some $j_{i},l_{i}\in \mathbb{N}$, $i\in \{1,\dots,k\}$.

Suppose that $j_{1}\leq \dots \leq j_{k}$. If $j_{1}$ is equal to $j_{k}$, we obtain that $\pm m^{l_{1}} \pm \dots \pm m^{l_{k}}=0$. If $j_{1} < j_{k}$, there is $t\in \{1,\dots,k\}$ such that $j_{1}\leq j_{2} \leq \dots \leq j_{t} < j_{t+1}=\dots = j_{k}$. Then, \[ \pm n^{j_{1}}m^{l_{1}} \pm \dots \pm n^{j_{t}}m^{l_{t}}= n^{j_{k}}( \pm m^{l_{k}} \pm \dots \pm m^{l_{t+1}}).\] If $\pm m^{l_{k}} \pm \dots \pm m^{l_{t+1}}$ is different from $0$, $n^{j_{k}}$ divides $\pm n^{j_{1}}m^{l_{1}} \pm \dots \pm n^{j_{t}}m^{l_{t}}$ which is not possible. Therefore, \[\pm m^{l_{k}} \pm \dots \pm m^{l_{t+1}} =0.\]
Assume that $l_{k}\leq \dots \leq l_{t+1}$. Then, \[ m^{l_{k}}(\pm 1 \pm m^{l_{k-1}-l_{k}} \pm \dots \pm m^{l_{t+1}-l_{k}})=0.\] Thus, $l_{k-1}-l_{k}=0$ and the sign of $\pm m^{l_{k-1}-l_{k}}$ is different from the one of $\pm 1$. At this point, the induction hypothesis can be used, so the statement also holds for $k$.

Recall what we have proved: if $x_{i_{1}}^{\pm 1} \dots x_{i_{k}}^{\pm 1} \in \ker f$, the number of $x_{i_{j}}$'s in the word $x_{i_{1}}^{\pm 1} \dots x_{i_{k}}^{\pm 1}$ is equal to the number of $x_{i_{j}}^{-1}$'s in $x_{i_{1}}^{\pm 1} \dots x_{i_{k}}^{\pm 1}$, for each $j\in \{1,\dots,k\}$. The last step is to show that this implies that $x_{i_{1}}^{\pm 1} \dots x_{i_{k}}^{\pm 1}$ is an element of $L^\prime$. We prove it by induction on $k$. The cases $k\in \{1,2\}$ are routine, so assume that the statement holds for $k-1$ and let us check it for $k$. By hypothesis,
\[ x_{i_{1}}^{\pm 1} \dots x_{i_{k}}^{\pm 1}= x_{i_{1}} \dots x_{i_{j-1}}^{\pm 1} x_{i_{1}}^{-1} x_{i_{j+1}}^{\pm 1}\dots x_{i_{k}}^{\pm 1} \text{ or } x_{i_{1}}^{\pm 1} \dots x_{i_{k}}^{\pm 1}= x_{i_{1}}^{-1} \dots x_{i_{j-1}}^{\pm 1}  x_{i_{1}} x_{i_{j+1}}^{\pm 1}\dots x_{i_{k}}^{\pm 1}.\]
If we denote $x_{i_{2}}^{\pm 1} \dots x_{i_{j-1}}^{\pm 1}$ by $w$, then we have that either
\[ x_{i_{1}}^{\pm 1} \dots x_{i_{k}}^{\pm 1}= x_{i_{1}} w x_{i_{1}}^{-1} w^{-1} w x_{j+1}^{\pm 1}\dots x_{i_{k}}^{\pm 1} \hspace{0.3cm} \text{or} \hspace{0.3cm}x_{i_{1}}^{\pm 1} \dots x_{i_{k}}^{\pm 1}= x_{i_{1}}^{-1} w x_{i_{1}} w^{-1} w x_{j+1}^{\pm 1}\dots x_{i_{k}}^{\pm 1}.\]
Observe that $x_{i_{1}}w x_{i_{1}}^{-1}w^{-1}$ ( or $x_{i_{1}}^{-1}w x_{i_{1}}w^{-1}$) is an element of $L^\prime$ and by the inductive hypothesis, $w x_{j+1}^{\pm 1}\dots x_{i_{k}}^{\pm 1} \in L^\prime$. In conclusion, $x_{i_{1}}^{\pm 1} \dots x_{i_{k}}^{\pm 1} \in L^\prime$.
\end{proof}

\section{Algorithmic problems}

In this section, we study algorithmic problems for finitely presented subdirect products of (some) groups in the class $\mathcal A$. Our approach follows closely the one in \cite[Section 7.1, Section 7.2]{Bridson3}.

Let $\mathcal{A^\prime} \subset \mathcal A$ be the subclass of groups that are CAT(0) and have unique roots. Note that, conjecturally, every group in the class $\mathcal A$ has a finite index subgroup that belongs to $\mathcal{A}^\prime$.

The class $\mathcal{A^\prime}$ contains all $2$-dimensional coherent RAAGs and more generally, all graphs of groups such that the underlying graph is a tree with free abelian vertex groups and cyclic edge groups.

We next proof that the multiple conjugacy problem is decidable for the class of finitely presented subgroups of the direct product of two groups in the class $\mathcal{A^\prime}$.

The \emph{multiple conjugacy problem} for a finitely generated group $G$ asks if there is an algorithm that, given an integer $l$ and two $l$-tuples of elements of $G$, say $x=(x_{1},\dots,x_{l})$ and $y=(y_{1},\dots,y_{l})$, can determine if there exists $g\in G$ such that $g x_{i} g^{-1}= y_{i}$ in $G$, for $i\in \{1,\dots,l\}$.

The solution to the multiple conjugacy problem for finitely presented residually free groups described in \cite[Section 7.1]{Bridson3} has two steps. In the first one, the authors give sufficient conditions for a subgroup of a bicombable group to have decidable multiple conjugacy problem. More precisely, they prove the following:

\begin{proposition}{\cite[Proposition 7.1]{Bridson3}}\label{Prop:bicombable}
Let $\Gamma$ be a bicombable group, let $H < \Gamma$ be a subgroup, and suppose that there exists a subgroup $L<H$ normal in $\Gamma$ such that $\Gamma \slash L$ is nilpotent. Then $H$ has a solvable multiple conjugacy problem.
\end{proposition}

Our main result, Theorem \ref{Main theorem}, states that if $S$ is a finitely presented subgroup of the direct product of two groups in the class, then $S$ has a finite index subgroup $S_0$ which is $H$-by-(free abelian). The above result is intended to prove that $S_0$ has decidable multiple conjugacy problem. In general, the decidability of the conjugacy problem does not pass from finite index subgroups to the group. However, it does if the group has unique roots.

\begin{lemma}{\cite[Lemma 7.2]{Bridson3}} \label{lem:roots}
Suppose $G$ is a group in which roots are unique and $H<G$ is a subgroup of finite index. If the multiple conjugacy problem for $H$ is solvable, then the multiple conjugacy problem for $G$ is solvable.
\end{lemma}

These two results are the main tools to prove the following:

\begin{theorem}\label{thm:conjugacy}
The multiple conjugacy problem is solvable in every finitely presented subgroup $S$ of the direct product of two groups in $\mathcal{A}^\prime$ ($S$ given by a finite presentation and an embedding to $G_1\times G_2$ where $G_i\in \mathcal{A}^\prime$ and the embedding is as a neat subdirect product).
\end{theorem}

\begin{proof}
Let $S< G_1\times G_2$ be a finitely presented neat subdirect product of $G_1, G_2\in \mathcal{A^\prime}$. That is, $L_i= S\cap G_i\ne \{1\}$ and the projection map $\pi_i\colon S \mapsto G_i$ is an epimorphism for $i\in \{1,2\}$. 

Now, the group $G_1\times G_2$ is CAT(0) since by assumption, both $G_1$ and $G_2$ are CAT(0). Furthermore, as shown in the proof of Theorem \ref{Main theorem}, we have that $(G_1\times G_2) \slash (L_1 \times L_2)$ is virtually abelian and so there is a finite index subgroup $G$ of $G_1\times G_2$ such that $G\slash (L_1\times L_2)$ is abelian. Furthermore, $G$ is CAT(0) for being a finite index subgroup of a CAT(0) group and so it is bicombable. Then, since $L_1\times L_2 < G\cap S < G$ and  $(G\cap S) \slash (L_1\times L_2)$ is abelian, it follows from Proposition \ref{Prop:bicombable} that $G\cap S$ has decidable multiple conjugacy problem. Now $S\cap G$ has finite index in $S$ and since $G_1 \times G_2$ has unique roots, so does $G$. Finally, from Lemma \ref{lem:roots} we conclude that $S$ has decidable multiple conjugacy problem.
\end{proof}

Notice that in Theorem \ref{thm:conjugacy} we require the finitely presented group $S$ to be given as a neat subdirect product of groups in the class $\mathcal{A^\prime}$. If we restrict to the family of $2$-dimensional coherent RAAGs, then given a finite presentation for $S$, one can effectively determine two $2$-dimensional coherent RAAGs $A_1$ and $A_2$, and an embedding $f\colon S \mapsto A_1 \times A_2$ such that $A_i \cap S \ne 1$ for $i\in \{1,2\}$ (see \cite{Ilya}, and also see \cite[Section 7.1]{Bridson3} to see how to relate the approaches). Then $\pi_i(f(S))=G_i$ is a finitely generated subgroup of $A_i$, where $\pi_i\colon A_1 \times A_2 \mapsto A_i$ is the natural projection map, $i\in \{1,2\}$, and from \cite{Kapovich} (see Lemma \ref{Lemma M} below), given a finite set of generators (which are the image of the generators of $S$), one can effectively describe the presentation of $G_i$, $i\in \{1,2\}$. Therefore, given a finite presentation of $S$, one can effectively determine $G_1 \times G_2$ so that $S$ is a neat subdirect product of $G_1\times G_2$. Furthermore, since finitely generated subgroups of coherent RAAGs are CAT(0) (see \cite[Corollary 9.5]{CDK20}) and have unique roots (see \cite{Duchamp}), we have that $G_1, G_2 \in \mathcal A^{\prime}$. Therefore, from Theorem \ref{thm:conjugacy} and the discussion above, we deduce the following:

\begin{corollary}
The multiple conjugacy problem is decidable for the class of finitely presented subgroups of the direct product of two $2$-dimensional coherent RAAGs.
\end{corollary}

Let us now focus on the \emph{membership problem}. Recall that the class $\mathcal G$ is defined as the class of cyclic subgroup separable graphs of groups with free abelian vertex groups and cyclic edge groups. We aim to show the following result:

\begin{theorem}\label{Theorem Membership}
If $S$ is a finitely presented subgroup of the direct product of two groups from the class $\mathcal G$ (given by a finite presentation and an embedding to $G_1\times G_2$, $G_i\in \mathcal G$) and $H \subseteq S$ is a finitely presentable subgroup (given by a finite generating set of words in the generators of $S$), then the membership problem for $H$ is decidable, i.e., there is an algorithm which, given $g\in S$ (as a word in the generators) will determine whether or not $g\in H$.
\end{theorem}

In \cite[Section 7.2]{Bridson3}, the authors prove that the membership problem is decidable for finitely presented subgroups of finitely presented residually free groups. The two key ingredients used in the proof are that limit groups have decidable membership problem (in fact, in \cite{W08}, Wilton proved that limit groups are subgroup separable) and the fact that presentations can be effectively described given a set of generators: if $\Gamma$ is a limit group over a free group, then there is an algorithm that, given a finite set $X \subseteq \Gamma$, will output a finite presentation for the subgroup generated by $X$ (see \cite[Lemma 7.5]{Bridson3}).

For groups in the class $\mathcal G$, these results also hold, namely, we have the following:

\begin{lemma}{\cite[Corollary 1.3]{Kapovich}}\label{Lemma M}
Let $G\in \mathcal G$. Then $G$ has solvable uniform membership problem. Moreover, there is an algorithm which, given a finite subset $X \subseteq G$, constructs a finite presentation for the subgroup $U= \langle X \rangle <G$.
\end{lemma}

The proof of Theorem~\ref{Theorem Membership} is now similar to the proof of \cite[Theorem K]{Bridson3} and we sketch it below. 

\begin{proof}[Proof of Theorem~\ref{Theorem Membership}]
Suppose that $S$ is a subgroup of $D= G_{1}\times G_{2}$, where $G_{1}$ and $G_{2}$ are in $\mathcal G$. A solution to the membership problem $H \subseteq D$ provides a solution for $H \subseteq S$. Define $L_{i}$ to be $H \cap G_{i}$ and $p_{i}$ to be the projection map $H \mapsto G_{i}$ for $i\in \{1,2\}$ and let $g= (g_{1},g_{2})\in G_{1}\times G_{2}$.

Suppose that some $L_{i}$ is trivial, say $L_{1}$. Then, $H$ is isomorphic to $p_{2}(H)$. In particular, $p_{2}(H)$ is finitely presented. By Lemma~\ref{Lemma M}, there is an algorithm that determines if $g_{2}$ lies in $p_{2}(H)$. If it does not, then $g\notin H$. If it does, we eventually find a word $w$ in the generators of $H$ so that $g
^{-1}w$ projects to $1$. Since $L_{1}= H \cap G_{1}$, we deduce that $g\in H$ if and only if $g^{-1}w=1$ and this equality can be checked because the word problem is solvable in residually finite groups and so in particular, in groups from $\mathcal G$.

It remains to consider the case when $L_{1}$ and $L_{2}$ are non-trivial. By Lemma~\ref{Lemma M}, we can determine algorithmically if $g_{i}\in H_{i}= p_{i}(H_{i})$ for $i\in \{1,2\}$. If $g_{i}\notin H_{i}$ for some $i\in \{1,2\}$, then $g\notin H$. Otherwise, we replace $G_{1}\times G_{2}$ by $H_{1}\times H_{2}$. We are now reduced to the case when $H$ is a full subdirect product. By Theorem~\ref{Main theorem}, $Q= D \slash L$ is virtually free abelian, where $L= L_{1}\times L_{2}$.

Let $\phi \colon D \mapsto Q$ be the quotient map. Virtually free abelian groups are subgroup separable, so if $\phi(g) \notin \phi(H)$, there is a finite quotient of $G$ that separates $g$ from $H$. But since $L= \ker \phi \subseteq H$, $\phi(g) \in \phi(H)$ if and only if $g\in H$. Thus an enumeration of the finite quotients of $D$ provides a procedure for proving that $g\notin H$ in this case. This terminates if $g\notin H$. We run this procedure in parallel with an enumeration of $g
^{-1}w$ that will terminate if $g\in H$.
\end{proof}

The algorithm described to solve the membership problem is not uniform (it depends on the subgroup). In \cite{BW08}, Bridson and Wilton showed that finitely presented residually free groups are actually subgroup separable providing a uniform algorithm to solve the membership problem. The RAAG $P_4$ is not subgroup separable (see \cite{NW00}), and so in general, coherent RAAGs are not subgroup separable. Therefore, one can not take this approach to obtain a uniform algorithm to solve the membership problem.

As we discussed above, since given a finitely presented subgroup of the direct product of two $2$-dimensional coherent RAAGs, one can effectively describe the RAAGs, we obtain the following:

\begin{corollary}
If $S$ is a finitely presented subgroup of the direct product of two $2$-dimensional coherent RAAGs (given by a finite presentation) and $H \subseteq S$ is a finitely presentable subgroup (given by a finite generating set of words in the generators of $S$), then the membership problem for $H$ is decidable, i.e., there is an algorithm which, given $g\in S$ (as a word in the generators) will determine whether or not $g\in H$.
\end{corollary}

\end{document}